\documentclass[12pt, letterpaper]{amsart}
\usepackage[left=1in,right=1in,bottom=0.5in,top=0.8in]{geometry}
\usepackage{amsfonts}
\usepackage{amsmath, amssymb}
\usepackage{graphicx}
\usepackage[font=small,labelfont=bf]{caption}
\usepackage{epstopdf}
\usepackage{xcolor}
\usepackage{amsthm}
\usepackage{float}
\usepackage{pgfplots}
\usepackage{listings}
\usepackage{longtable}
\usepackage{mathrsfs}
\usepackage{bbold}
\usepackage{comment}
\usepackage{esint}
\usetikzlibrary{arrows, patterns, patterns.meta, calc, math}

\usepackage{mathtools}
\usepackage{scalerel}[2014/03/10]
\usepackage[usestackEOL]{stackengine}
\usepackage[utf8]{inputenc}
\usepackage[T1]{fontenc}
\usepackage[shortlabels, inline]{enumitem}
\DeclarePairedDelimiterX\set[1]\lbrace\rbrace{#1}

\allowdisplaybreaks

\usepackage{scalerel}[2014/03/10]
\usepackage[usestackEOL]{stackengine}

\usepackage[backref=page]{hyperref}
\hypersetup{
plainpages=false,
colorlinks,
linkcolor={cyan!90!black},
citecolor={magenta},
urlcolor={red!90!black},
bookmarksdepth=2,}

\hypersetup{
pdftitle={On Hardy-Littlewood-Sobolev estimates for weighted Laplacians},
pdfsubject={Mathematics, PDE, Analysis},
pdfauthor={Khalid Baadi},
pdfkeywords={Parabolic systems, Cauchy problems, Green operators, variational methods, Gaussian estimates, Heat Kernel, Energy equality.}
}

\newtheorem{thm}{Theorem}[section]
\newtheorem{cor}[thm]{Corollary}
\newtheorem{prop}[thm]{Proposition}
\newtheorem{lem}[thm]{Lemma}

\theoremstyle{definition}

\theoremstyle{remark}
\newtheorem{rem}[thm]{Remark}

\definecolor{energy}{RGB}{114,0,172}
\definecolor{freq}{RGB}{45,177,93}
\definecolor{spin}{RGB}{251,0,29}
\definecolor{signal}{RGB}{203,23,206}
\definecolor{circle}{RGB}{217,86,16}
\definecolor{average}{RGB}{203,23,206}

\definecolor{pa}{rgb}   {.4, .4, 0} 
\definecolor{kb}{rgb}   {.6, 0, 0}

\colorlet{shadecolor}{gray!20}
\pgfplotsset{compat=1.9}

\usepgflibrary{fpu}
\makeatletter
\let\c@equation\c@thm
\makeatother
\numberwithin{equation}{section}

\author{Pascal Auscher}
\address{Universit{\'e} Paris-Saclay, CNRS, Laboratoire de Math\'{e}matiques d'Orsay, 91405 Orsay, France}
\email{pascal.auscher@universite-paris-saclay.fr}
\author{Khalid Baadi}
\address{Universit{\'e} Paris-Saclay, CNRS, Laboratoire de Math\'{e}matiques d'Orsay, 91405 Orsay, France}
\email{khalid.baadi@universite-paris-saclay.fr}
\keywords{Degenerate Laplacians, Sobolev embeddings, heat kernel, Muckenhoupt weights, weighted Lorentz spaces.}
\date{April 16, 2026}
\subjclass[2010]{Primary: 46E35, 35K08, 42B37 Secondary: 35K65, 42B20.}
\title[On Hardy-Littlewood-Sobolev estimates for degenerate Laplacians]{On Hardy-Littlewood-Sobolev estimates for degenerate Laplacians}

\DeclareMathOperator*{\supess}{ess\,sup}

\begin{document}

\begin{abstract}
We establish  norm inequalities for fractional powers of degenerate Laplacians, with degeneracy being determined by weights in the Muckenhoupt class $A_2(\mathbb{R}^n)$, accompanied by specific additional reverse H\"older assumptions. This extends the known results for classical Riesz potentials. The approach is based on size estimates for the degenerate heat kernels. The approach also applies to more general weighted degenerate operators.
\end{abstract}
\maketitle

\tableofcontents

\section{Introduction}\label{section 1}

Degenerate elliptic operators  where the degeneracy is controlled by weights in the Muckenhoupt class  have been introduced in \cite{MR643158}. Since then, they have attracted attention with extension to parabolic equations \cite{MR772255,MR748366, MR799906, MR1711391}. For $A_2(\mathbb{R}^n)$ weights $\omega$, $-\Delta_\omega$ can be defined as  a maximal accretive operator on the weighted space $L^2_\omega(\mathbb{R}^n)$.  Much is known. For example, the associated semigroup  is bounded on $L^p_\omega(\mathbb{R}^n)$ for all $1\le p \le \infty$, it  has $H^\infty$ functional calculus (equivalently, Borel functional calculus) if $1<p<\infty$.  The particular study of its square root can be found in   \cite{cruz2015kato, cruz2017kato} with a proof of the Kato problem on $L^2_\omega(\mathbb{R}^n)$ and extension to some range of $p$, and even a "weighted theory" with weights in the Muckenhoupt classes for the measure $w(x)\, dx$.

It is not known whether there are $L^p-L^q$ bounds for the fractional powers $(-\Delta_\omega)^{-\alpha}$ in the spirit of the Hardy-Littlewood-Sobolev estimates as in the case where $\omega=1$. Recall that a possible way to obtain such estimates when $\omega=1$ is to start from  the heat semigroup $L^p-L^q$ bounds for  $q>p$, and use this in a Calder\'on's representation formula to prove  weak type $(p,q)$ bounds for the fractional powers $(-\Delta)^{-\alpha}$ with the relation between $\alpha, p, q$ and dimension obeying the rule imposed by scaling, and then proceed by Marcinkiewicz interpolation.  However, this strategy does not apply as such in the weighted case as the semigroup is not bounded from $L^p_\omega(\mathbb{R}^n)$ to $L^q_\omega(\mathbb{R}^n)$ if $q>p$. Our observation is that when $\omega$ additionally belongs to the reverse H\"older class $RH_{\frac{q}{p}}(\mathbb{R}^n)$ then the correct boundedness property is from  $L^p_\omega(\mathbb{R}^n)$ to $L^q_{{\omega}^{q/p}}(\mathbb{R}^n)$. Hence, the above strategy can be developed for the fractional powers and this allows us to prove the following result.  Definitions of weighted spaces and operators are in the next section.

\begin{thm}\label{thm:principal}
    Let $n\ge 1$ and $\omega \in A_2(\mathbb{R}^n)$. Let $\Delta_\omega$ be the weighted Laplacian defined by the method of forms. Let $1<p<q<\infty$. Assume that $\omega \in RH_{\frac{q}{p}}(\mathbb{R}^n)$ and set $\alpha=\frac{n}{2}(\frac{1}{p}-\frac{1}{q})$. Then, there exists a constant $C = C([ \omega  ]_{A_2},[ \omega  ]_{RH_{\frac{q}{p}}},n,p,q)<\infty$ such that
    \begin{equation}\label{eq:fractionalHLS}
       \| \omega^{1/p}f \|_{L^q_{\mathrm{d}x}(\mathbb{R}^n)}= \| f \|_{L^q_{{\omega}^{q/p}}(\mathbb{R}^n)} \leq C \| (-\Delta_\omega)^\alpha f \|_{L^p_{\omega}(\mathbb{R}^n)}, \quad \forall f \in D((-\Delta_\omega)^\alpha).
    \end{equation}
    In particular, $(-\Delta_\omega)^{-\alpha}$ defines a bounded operator from $L^p_\omega(\mathbb{R}^n)$ to $L^q_{{\omega}^{q/p}}(\mathbb{R}^n)$ or, equivalently, $w^{1/p}(-\Delta_\omega)^{-\alpha}$ is bounded from $L^p_\omega(\mathbb{R}^n)$ to $L^q_{\mathrm{d}x}(\mathbb{R}^n)$. If $p=1$, then there exists a constant $C = C([ \omega  ]_{A_2},[ \omega  ]_{RH_{q}},n,q)<\infty$ such that
    \begin{equation}\label{eq:fractionalHLSweak}
        \| \omega  f \|_{L^{q,\infty}_{\mathrm{d}x}(\mathbb{R}^n)} \leq C \| (-\Delta_\omega)^\alpha f \|_{L^1_{\omega}(\mathbb{R}^n)}, \quad \forall f \in D((-\Delta_\omega)^\alpha).
    \end{equation}
    In particular, $\omega \,(-\Delta_\omega)^{-\alpha}$ defines a bounded operator from $L^1_\omega(\mathbb{R}^n)$ to $L^{q,\infty}_{\mathrm{d}x}(\mathbb{R}^n)$.
\end{thm}

The expert reader may find the statement for $(-\Delta_\omega)^{-\alpha}$ alike to the famous Muckenhoupt-Wheeden weighted estimates \cite{muckenhoupt1974weighted}: under the same conditions on $p,q,\alpha$ and dimension, the unweighted fractional Laplacian $(-\Delta)^{-\alpha}$ is bounded from  $L^p_{\tilde\omega^p}(\mathbb{R}^n)$ to $L^q_{\tilde\omega^q}(\mathbb{R}^n)$ if and only if $\tilde\omega \in A_{p,q}(\mathbb{R}^n)$. There is indeed the same rule for the change of  weight. However, aside from the $A_2$ condition required to define our degenerate Laplacian, we  have a weaker condition on $\omega$ compared to  $\tilde\omega^p$.   Indeed, the reverse H\"older condition  $\tilde w^p\in RH_{\frac{q}{p}}(\mathbb{R}^n)$ holds when $\tilde\omega \in A_{p,q}(\mathbb{R}^n)$, but the converse is not true as we shall see. Besides, since $\omega$ is used to construct  the weighted Laplacian $\Delta_\omega$ on $L^2_{\omega}(\mathbb{R}^n)$, this is not really a weighted estimate that we are proving.

Inequalities \eqref{eq:fractionalHLS} and \eqref{eq:fractionalHLSweak} are  a priori weighted Sobolev embeddings of a functional nature. Some more geometric weighted Sobolev-Gagliardo-Nirenberg inequalities were obtained in  \cite{lacey2010sharp}. Namely,   for all $1\le p<q<\infty$ with $\frac 1 2 =\frac{n}{2}(\frac{1}{p}-\frac{1}{q})$, $\| f \|_{L^q_{{\tilde\omega}^{q}}(\mathbb{R}^n)} \lesssim \| \nabla f \|_{L^p_{\tilde\omega^p}(\mathbb{R}^n)}$ when $\tilde\omega \in A_{p,q}(\mathbb{R}^n)$ (which, again, requires $\tilde\omega^p \in RH_{\frac{q}{p}}(\mathbb{R}^n)$). 

We can only compare these results with ours  when $\alpha=1/2$ in the range of $p$ for which $ \| \nabla f \|_{L^p_{\omega}(\mathbb{R}^n)} $ dominates or is dominated by $ \| (-\Delta_\omega)^{1/2} f \|_{L^p_{\omega}(\mathbb{R}^n)}$. When $p=2$, one has $\| (-\Delta_\omega)^{1/2} f \|_{L^2_{\omega}(\mathbb{R}^n)}= \| \nabla f \|_{L^2_{\omega}(\mathbb{R}^n)}$ by construction. Moreover, deducing our results from these or other known results (\textit{e.g.}, \cite{muckenhoupt1974weighted}), or conversely,  requires additional restrictions on the class of weights and on the dimension, as we shall explain.

As a matter of fact, we will show that our reverse H\"older assumption is sharp in the sense that it is implied by \eqref{eq:fractionalHLS}. At the level of the semigroup, we shall see that with the assumptions of the theorem, then $\omega \in RH_{\frac{q}{p}}(\mathbb{R}^n)$ is equivalent to  $e^{t\Delta_w}:  L^p_{\omega}(\mathbb{R}^n) \to L^q_{{\omega}^{q/p}}(\mathbb{R}^n)$
with bound on the order $t^{-\alpha}$ for all $t>0$. This polynomial bound is the key point of our method.

Our method also allows for various generalizations. In \cite{O'Neil1963convolution}, O'Neil showed that the classical unweighted $L^p-L^q$ estimates for Riesz potentials (that is, the fractional Laplacians of negative order) can be improved to $L^p-L^{q,p}$ estimates, where $L^{q,p}$ is a Lorentz space, using the translation invariance of the operators. Later, it was observed by Tartar (see \cite{MR1662313}) that this is a consequence of real interpolation more than translation invariance.  Here too, real interpolation applies and yields two  apparently distinct improvements for the weighted fractional Laplacian, see Theorem \ref{thm:principal*} and Theorem \ref{cor:principal**}, in agreement with observations made in \cite{MR4440073}. In particular, we shall show that $w^{\frac 1 p}(-\Delta_\omega)^{-\alpha}$ is bounded from $L^p_{\omega}(\mathbb{R}^n)$ to the unweigthed Lorentz space $L^{q,p}(\mathbb{R}^n)$ under the same assumption as for \eqref{eq:fractionalHLS}.

Incorporating results on bounded holomorphic functional calculus, fractional powers $z^{-\alpha}$ can be replaced by holomorphic functions in sectors with the same behavior at the origin and infinity.

Given that our methods rely solely on functional calculus and upper bounds for the kernel of the semigroup $e^{t\Delta_w}$, our results extend to degenerate second-order elliptic operators whose semigroups satisfy Gaussian upper bounds, including those with real coefficients (see Section \ref{section: coeff}). For non real coefficients, we think that the result should hold in limited range for the exponents $p$ and $q$, but our arguments no longer seem to apply even though, again, non-optimal results can still be obtained using the literature, as we shall see.

As a motivation for this work, bounds for fractional weighted Laplacians are used to construct solutions to second-order degenerate parabolic equations with unbounded lower-order terms \cite{baadi2026wellposedness}. They are also used to establish a boundedness result for degenerate parabolic Riesz transforms. The second author will address this in a forthcoming work.

\subsubsection*{\textbf{Notation}} Throughout this paper, we adopt the following notation.
\begin{enumerate}[label=$\blacklozenge$]
\item We fix an integer $n \ge 1$. 
\item We use the notation $\mathrm{d}x$ for the Lebesgue measure on $\mathbb{R}^n$.
\item For any Lebesgue measurable set $E \subset \mathbb{R}^n$, we denote its Lebesgue measure by $|E|$.
\item By convention, the notation $C(a,b,\ldots)$ denotes a constant depending only on the parameters $(a,b,\ldots)$.
\item All cubes $Q \subset \mathbb{R}^n$ considered in this paper are assumed to have sides parallel to the coordinate axes.
\end{enumerate}

\subsubsection*{\textbf{Acknowledgements}}
We want to thank Chema Martell for bringing our attention to \cite{MR4440073} concerning Lorentz spaces after we completed this work and Carlos Pérez for enlightening discussions. We also thank the anonymous referee for providing suggestions that enhanced the presentation.

\section{Preliminaries and basic assumptions}\label{section 2}

\subsection{Muckenhoupt weights}

We say that an almost everywhere (for Lebesgue measure) defined function $\omega : \mathbb{R}^n \rightarrow (0,\infty]$ is a weight if it is locally integrable on $\mathbb{R}^n$. The Muckenhoupt class $A_p(\mathbb{R}^n)$, $p\in (1,\infty)$, is defined as the set of all weights $\omega$  verifying 
\begin{equation}\label{MuckWeight}
      [ \omega  ]_{A_p}:= \sup_{Q \subset \mathbb{R}^n  } \left ( \frac{1}{|Q|} \int_Q \omega(x) \ \mathrm{d}x   \right ) \left ( \frac{1}{|Q|} \int_Q \omega^{1-p'}(x) \ \mathrm{d}x   \right )^{p-1} < \infty,
\end{equation}
where the supremum is taken with respect to all cubes $Q \subset \mathbb{R}^n$. For $p=1$, the Muckenhoupt class $A_1(\mathbb{R}^n)$ is defined as the set of all weights $\omega$  verifying
\begin{equation*}
    \frac{1}{|Q|} \int_Q \omega(x) \, \mathrm{d}x \leq C \omega(y), \quad \text{for almost every} \ y \in Q,
\end{equation*}
for some constant $C <\infty$ and for all cubes $Q \subset \mathbb{R}^n$. Finally, we define $A_\infty(\mathbb{R}^n):= \bigcup_{p\ge1} A_p(\mathbb{R}^n)$.

The reverse H\"older class $RH_q(\mathbb{R}^n)$, $q\in (1,\infty)$, is defined as the set of all weights $\omega$ verifying 
\begin{equation*}
    \left ( \frac{1}{|Q|} \int_Q \omega^q(x) \ \mathrm{d}x   \right )^{\frac{1}{q}} \leq \frac{C}{|Q|} \int_Q \omega(x) \ \mathrm{d}x,
\end{equation*}
for some constant $C <\infty$ and for all cubes $Q \subset \mathbb{R}^n$. The infimum of such constants $C$ is what we denote by $[ \omega  ]_{RH_q}$. For $q=\infty$, the reverse H\"older class $RH_\infty(\mathbb{R}^n)$ is defined as the set of all weights $\omega $ verifying 
\begin{equation*}
    \omega(y) \leq  \frac{C}{|Q|} \int_Q \omega(x) \ \mathrm{d}x, \quad \text{for almost every} \ y \in Q,
\end{equation*}
for some constant $C <\infty$ and for all cubes $Q \subset \mathbb{R}^n$.

The Muckenhoupt-Wheeden class $A_{p,q}(\mathbb{R}^n)$, $1\le p \le q <\infty$, consists of all weights $\omega$ satisfying
\begin{equation*}
    [ \omega  ]_{A_{p,q}}:= \sup_{Q \subset \mathbb{R}^n  } \left ( \frac{1}{|Q|} \int_Q \omega^q(x) \ \mathrm{d}x   \right )^{\frac{1}{q}} \left ( \frac{1}{|Q|} \int_Q \omega^{-p'}(x) \ \mathrm{d}x   \right )^{\frac{1}{p'}} < \infty,
\end{equation*}
when $1<p$, and 
\begin{equation*}
   \left ( \frac{1}{|Q|} \int_Q \omega^q(x) \ \mathrm{d}x   \right )^{\frac{1}{q}} \leq C \omega(y), \quad \text{for almost every} \ y \in Q,
\end{equation*}
for some constant $C <\infty$ and for all cubes $Q \subset \mathbb{R}^n$ if $p=1$.

In the following proposition, we summarize several properties of Muckenhoupt weights. For the proofs, except for $(8)$ which is a new  formulation that is more convenient for our purpose, refer to \cite{garcia2011weighted} and \cite[Chapter 7]{grafakos2008classical}.

\begin{prop}\label{prop:weights} We have the following properties.
    \begin{enumerate}
        \item $A_p(\mathbb{R}^n) \subset A_q(\mathbb{R}^n)$ for all $1 \leq p \leq q \leq \infty$.
        \item $RH_q(\mathbb{R}^n) \subset RH_p(\mathbb{R}^n)$ for all $1 < p \leq q \leq \infty$.
        \item If $\omega \in A_p(\mathbb{R}^n)$, $1<p<\infty$, then there exists $q \in (1,p)$ such that $\omega \in A_q(\mathbb{R}^n)$.
        \item If $\omega \in RH_q(\mathbb{R}^n)$, $1<q<\infty$, then there exists $p \in (q,\infty)$ such that $\omega \in RH_p(\mathbb{R}^n)$.
        \item $A_\infty(\mathbb{R}^n)= \bigcup_{1\le p<\infty} A_p(\mathbb{R}^n)= \bigcup_{1< q\le \infty} RH_q(\mathbb{R}^n) $.
        \item If $1\leq p \leq \infty$ and $1<q<\infty$, then $\omega \in A_p(\mathbb{R}^n) \cap RH_q(\mathbb{R}^n)$ if and only if $\omega^q \in A_{q(p-1)+1}(\mathbb{R}^n)$.
        \item If $1 \le p \le q < \infty$, then $\omega \in A_{p,q}(\mathbb{R}^n)$ if and only if $\omega^q \in A_{1+q/p'}(\mathbb{R}^n)$ if and only if $\omega \in A_{1+1/p'}(\mathbb{R}^n)\cap RH_q(\mathbb{R}^n)$.
        \item If $1 \le  p < q < \infty$, then $\omega^{1/p} \in A_{p,q}(\mathbb{R}^n)$ if and only if 
        $\omega \in A_p(\mathbb{R}^n) \cap RH_{q/p}(\mathbb{R}^n)$.
    \end{enumerate}
\end{prop}

\begin{proof} As we said, we only prove (8).
By (6), the condition $\omega \in A_p(\mathbb{R}^n) \cap RH_{q/p}(\mathbb{R}^n)$ is equivalent to $\omega^{{q}/{p}} \in A_{\frac{q}{p}(p-1)+1}(\mathbb{R}^n)$, which can be rewritten as $(\omega^{{1}/{p}})^q \in A_{1+{q}/{p'}}(\mathbb{R}^n)$. By (7), this is equivalent  to the fact that $\omega^{{1}/{p}} \in A_{p,q}(\mathbb{R}^n)$.
 \end{proof}

\begin{rem}
   It is well known (see, for instance, \cite[Example 7.1.7]{grafakos2008classical}) that a weight of the form $\omega(x) = |x|^\beta$ belongs to the Muckenhoupt class $A_p(\mathbb{R}^n)$ for $p > 1$ if and only if $-n < \beta < n(p-1)$. Furthermore, according to point (6) of Proposition \ref{prop:weights} above, a weight $\omega$ belongs to the intersection $A_2(\mathbb{R}^n) \cap RH_{\frac{q}{p}}(\mathbb{R}^n)$ for $p < q$ if and only if $\omega^{\frac{q}{p}} \in A_{1+\frac{q}{p}}(\mathbb{R}^n)$. As a result, a typical example of a weight $\omega$ considered in the statements of the results in this paper, satisfying $\omega \in A_2(\mathbb{R}^n) \cap RH_{\frac{q}{p}}(\mathbb{R}^n)$ with $p < q$, is given by
    \begin{equation*}
        \omega(x) = |x|^\beta, \quad \forall x \in \mathbb{R}^n, \quad \text{with} \ -n\, \frac{p}{q} < \beta < n.
    \end{equation*}

\end{rem}

\subsection{Function spaces}

We refer to \cite[Ch. V]{Stein1993_HA}, \cite{garcia2011weighted} and \cite[Chapter 7]{grafakos2008classical} for general background and for the proofs of all the results concerning weights that we will cite below.

\textbf{Throughout this paper, $\omega$ denotes a fixed weight belonging to the Muckenhoupt class $A_2(\mathbb{R}^n )$}. By definition \eqref{MuckWeight} with $p=2$, the weight $\omega^{-1}$ is also in the Muckenhoupt class $A_2(\mathbb{R}^n )$ with $  [ \omega^{-1}  ]_{A_2}=  [ \omega  ]_{A_2}$. We introduce the measure $\mathrm d \omega := \omega(x) \mathrm d x$ and if $E \subset \mathbb{R}^n$ a Lebesgue measurable set, we write $\omega(E)$ instead of $\int_E \mathrm d \omega$. For all $x \in \mathbb{R}^n$ and $r>0$, we set $\omega_r(x):=\omega(B(x,\sqrt{r}))$.

It follows from \eqref{MuckWeight} with $p=2$ that there exits constants $\eta \in (0,1)$ and $\beta >0$, depending only on $n$ and $ [ \omega  ]_{A_2}$, such that 
\begin{equation}\label{MuckProportion}
\beta^{-1}\left ( \frac{\left | E \right |}{\left | Q \right |} \right )^{\frac{1}{2\eta }}\leq \frac{\omega(E)}{\omega(Q)}\leq \beta\left ( \frac{\left | E \right |}{\left | Q \right |} \right )^{2\eta},
\end{equation}
whenever $Q \subset \mathbb{R}^n$ is a cube and for all measurable sets $E \subset Q$. In particular, there exists a constant $D$, depending only on $n$ and $ [ \omega  ]_{A_2}$, called the doubling constant for $\omega$ such that
\begin{equation}\label{DoublingMuck}
    \omega(2Q)\leq D \omega(Q),
\end{equation}
for all cubes $Q \subset \mathbb{R}^n$. We may replace cubes by Euclidean balls. For simplicity, we keep using the same notation and constants.

For every $p\ge 1$ and $K \subset \mathbb{R}^n$ a measurable set, we let $L^p_\omega(K)$ be the space of all measurable functions $f:K \rightarrow \mathbb{C}$ such that $$\|f\|_{L^p_\omega(K)}:=\left ( \int_K |f|^p \mathrm d \omega  \right )^{1/p}<\infty.$$
In particular, $L^2_\omega(\mathbb{R}^n)$ is the Hilbert space of square-integrable functions on $\mathbb{R}^n$ with respect to $\mathrm{d}\omega$. We denote its norm by ${\lVert \cdot \rVert}_{2,\omega}$ and its inner product by $\langle \cdot , \cdot \rangle_{2,\omega}$.  The class $\mathcal{D}(\mathbb{R}^n)$ consists of smooth ($C^\infty$), compactly supported test functions on $\mathbb{R}^n$, is dense in $L^2_{\omega}(\mathbb{R}^n)$ as it is dense in $C_c(\mathbb{R}^n)$, the space of continuous functions on $\mathbb{R}^n$ with compact support, and this latter space is dense in $L^2_{\omega}(\mathbb{R}^n)$ as $\mathrm d \omega$ is a Radon measure on $\mathbb{R}^n$. Moreover, using the condition \eqref{MuckWeight} with $p=2$, we have 
\begin{equation*}
L^2_{\omega}(\mathbb{R}^n) \subset L^1_{\text{loc}}(\mathbb{R}^n, \mathrm{d}x).
\end{equation*}

\subsection{The degenerate Laplacian}

We define $H^1_{\omega}(\mathbb{R}^n)$ as the space of functions $f \in L^2_\omega(\mathbb{R}^n)$ for which the distributional gradient $\nabla_x f$ belongs to $L^2_\omega(\mathbb{R}^n)^n$, and equip this space with the norm
$\left\| f \right\|_{H^1_\omega} := ( \left\| f \right\|_{2,\omega}^2 + \left\| \nabla_x f \right\|_{2,\omega}^2 )^{1/2}$ making it a Hilbert space. The class $\mathcal{D}(\mathbb{R}^n)$ is dense in $H^1_{\omega}(\mathbb{R}^n)$ and this follows from standard truncation and convolution techniques combined with the boundedness of the maximal operator on $L^2_\omega(\mathbb{R}^n)$. For a proof, see \cite[Thm. 2.5]{kilpelainen1994weighted}.

We define $-\Delta_\omega$ as the unbounded self-adjoint operator on $L^2_\omega(\mathbb{R}^n)$ associated to the positive symmetric sesquilinear form on $H^1_\omega(\mathbb{R}^n) \times H^1_\omega(\mathbb{R}^n)$ defined by
\begin{equation*}
    (u,v) \mapsto \int_{\mathbb{R}^n} \nabla_x u \cdot \overline{\nabla_x v}  \ \mathrm d \omega.
\end{equation*}
The operator $-\Delta_\omega$ is injective as the measure $\mathrm{d}\omega$ has infinite mass by \eqref{MuckProportion}. For all $\beta \in \mathbb{R}$, we define $(-\Delta_\omega)^\beta$ as the self-adjoint operator $\mathbf{t^{\beta}}(-\Delta_\omega)$, that is, the operator obtained by applying the function $f_\beta(z)=z^\beta$ to $-\Delta_\omega$. This construction can be made either via the functional calculus for sectorial operators or, equivalently, via the Borel functional calculus, since $-\Delta_\omega$ is self-adjoint. For further details, we refer to \cite{reed1980methods, McIntosh86, haase2006functional}. Let the domain of $(-\Delta_\omega)^\beta$ be denoted by $D((-\Delta_\omega)^\beta)$. Note that 
if $\beta>0$, then $D((-\Delta_\omega)^{-\beta})$ is the range of $(-\Delta_\omega)^\beta$ from its domain.

The operator $-\Delta_\omega$ generates a contractive holomorphic semi-group $(e^{z\Delta_\omega})_{z \in S_{\frac{\pi}{2}}}$ on the half-plane $S_{\frac{\pi}{2}}:=\left\{ z \in \mathbb{C}\setminus \{0\} \, : |\mathrm{arg}(z)| < \frac{\pi}{2} \right\} = \left\{ z \in \mathbb{C} \, : \mathrm{Re}(z) > 0 \right\}$. We now present the following lemma, which provides the Gaussian upper and lower  bounds  for the positive-time heat kernels. These estimates are well known and we refer to \cite{cruz2014corrigendum, ataei2024fundamental, baadi2025degenerate} by setting $A = I_n $ in those references. 

\begin{lem}\label{lem:GULB}
For all $t>0$ and $f \in L^2_\omega(\mathbb{R}^n)$, we have for almost every $x\in \mathbb{R}^n$,
    \begin{equation*}
        e^{t\Delta_\omega}f(x)= \int_{\mathbb{R}^n} K_t(x,y)f(y) \, \mathrm{d}\omega(y),
    \end{equation*}
    where $K_t(x,y)=\Gamma(t,x;0,y)$ and $\Gamma(t,x;s,y)$ is the generalized fundamental solution for the weighted heat operator $\partial_t-\Delta_\omega$. Moreover, there exist two constants $1<C=C([ \omega  ]_{A_2},n)<\infty$ and $1<c=c([ \omega  ]_{A_2},n)<\infty$ such that 
\begin{equation}\label{eq:GULB}
     \frac{C^{-1}}{\min(\omega_{t}(x),\omega_t(y))} e^{- \frac{c\left | x-y \right |^2}{t}} \leq 
     K_t(x,y) \leq \frac{C}{\max(\omega_{t}(x),\omega_t(y))} e^{- \frac{\left | x-y \right |^2}{ct}},
\end{equation}
for all $t>0$ and for (almost) all $(x,y) \in \mathbb{R}^{2n}$. 
\end{lem}

Note that if $c>0$,  $(t,x,y)\to \frac{\max(\omega_{t}(x),\omega_t(y))}  {\min(\omega_{t}(x),\omega_t(y))} e^{- \frac{c\left | x-y \right |^2}{t}}$ is uniformly bounded,  hence having the mininum, the maximum or a multiplicative convex combination   $\omega_{t}(x)^\gamma\omega_t(y)^{1-\gamma}$ at the denominator can be used interchangeably with different constants $C,c$. 
     
We remark that there are also  H\"older continuity bounds for $K_t(x,y)$ but we do not need them. 

Following the standard argument in \cite[Lemma~19,~p.~48]{auscher1998square} (more precisely its adaptation to the weighted setting in \cite[Theorem 4]{cruz2014corrigendum}), one derives the following Gaussian estimates for complex-time and time derivatives of the heat kernels.

\begin{cor}\label{cor:semigroupe}
    For all $z \in \mathbb{C}$ with $\mathrm{Re}(z)>0$ and $f \in L^2_\omega(\mathbb{R}^n)$, we have for almost every $x\in \mathbb{R}^n$,
    \begin{equation*}
        e^{z\Delta_\omega}f(x)= \int_{\mathbb{R}^n} K_z(x,y)f(y) \, \mathrm{d}\omega(y).
    \end{equation*}
    Moreover, for all $\mu \in (0,\frac{\pi}{2})$,  there are constants $C=C([ \omega  ]_{A_2},n,\mu)<\infty$ and $c=c([ \omega  ]_{A_2},n,\mu)>0$ such that
    \begin{equation}\label{eq:GULBComplexe}
     | K_z(x,y) | \leq \frac{C}{\max(\omega_{|z|}(x),\omega_{|z|}(y))} e^{- \frac{\left | x-y \right |^2}{c|z|}}.
\end{equation}
for all $z \in S_\mu:= \left\{ z \in \mathbb{C} \setminus {0} , : |\mathrm{arg}(z)| < \mu \right\}$ and (almost) all $(x,y) \in \mathbb{R}^{2n}$. {Furthermore, for all $k \ge 1$ and all $t > 0$, the operator $(t\Delta_\omega)^k e^{t\Delta_\omega} = t^k \frac{\mathrm{d}^k}{\mathrm{d}t^k} e^{t\Delta_\omega}$ is an integral operator whose kernel satisfies Gaussian upper bounds of the same form as in \eqref{eq:GULBComplexe}, with $t$ replacing $|z|$.}
\end{cor}

With these results, we can deduce the following (classical) consequences.

\begin{prop}\label{prop:boundednessandconvergence}
We have the following statements
\begin{enumerate} 
    \item Let $1\le p\le \infty$ and $k\in \mathbb{N}$.  Then $t^k \frac{\mathrm{d}^k}{\mathrm{d}t^k} e^{t\Delta_\omega}$ extends to a bounded operator on $L^p_\omega(\mathbb{R}^n)$ with uniform bound with respect to $t>0$.
    \item Let $1\le p<\infty$ and $k\in \mathbb{N}$. Then for all $f\in L^p_\omega(\mathbb{R}^n)$, $t^k \frac{\mathrm{d}^k}{\mathrm{d}t^k} e^{t\Delta_\omega}f$ converges in $L^p_\omega(\mathbb{R}^n)$ 
    to $f$   when $t\to 0$ and $k=0$, to 0 when $t\to 0$ and {$k\ge 1$} and to 0 when $t\to \infty$ and {$k\ge 0$}. 
    \end{enumerate}
\end{prop}

\begin{proof}
    The first item is a classical consequence of the Schur Lemma together with 
    \begin{equation*}
        \supess_x \int_{\mathbb{R}^n} \bigg|t^k \frac{\mathrm{d}^k}{\mathrm{d}t^k} K_t(x,y)\bigg| \, \mathrm{d}\omega(y) + \supess_y \int_{\mathbb{R}^n} \bigg|t^k \frac{\mathrm{d}^k}{\mathrm{d}t^k} K_t(x,y)\bigg| \, \mathrm{d}\omega(x) \le C_k
    \end{equation*}
    that follow from the Gaussian upper bounds. 
    The convergences come from the construction of the semigroup when $p=2$. They extend to all $1\le p<\infty$ by standard arguments using the Gaussian upper bounds. 
\end{proof}

More properties, such as the boundedness of the $H^\infty$-calculus of $-\Delta_\omega$ and of its Riesz transform on $L^p_\omega(\mathbb{R}^n)$ for $p \neq 2$, also hold \cite{cruz2017kato}. However, we point out that the above results are all we need for the proof of Theorem \ref{thm:principal}.

\section{Hardy-Littlewood-Sobolev estimates}

In this section, we begin with key lemmas that play a central role in   the proof of Theorem \ref{thm:principal} that we present next. We then discuss several extensions and refinements: first, we extend the result to weighted Lorentz spaces; second, to the functional calculus; and eventually, we show that the condition imposed on the weight is sharp.

\subsection{Key lemmas}
Recall that the semigroup extends to a bounded operator on $L^p_\omega(\mathbb{R}^n)$ for all $1 \leq p \leq \infty$. 
We wish to obtain "$L^p-L^q$ bounds" when $p<q$. We record the following key lemma on the boundedness of the semigroup, with a bound involving a suitable power of $t$, provided the weight $\omega$ lies in a reverse H\"older class. We shall see that the open ended property of the reverse H\"older condition assumed in this proof will be used in an essential way.
\begin{lem}\label{lem:HLS}
    Let $1 < p < q<\infty$. Assume that $\omega \in RH_{\frac{q}{p}}(\mathbb{R}^n)$. Then, there exist two constants $C = C([ \omega  ]_{A_2},[ \omega  ]_{RH_{\frac{q}{p}}},n,p,q)<\infty$ and $\varepsilon_{\omega, p,q}=\varepsilon(  [ \omega  ]_{RH_{\frac{q}{p}}},p,q,n)>0$ such that 
    \begin{equation*}
        \| e^{t\Delta_\omega} f \|_{L^q_{{\omega}^{q/p+\varepsilon}}(\mathbb{R}^n)} \leq C t^{-\frac{n}{2}(\frac{1}{p}-\frac{1}{q})} \|  f \|_{L^p_{\omega^{1+(p/q)\varepsilon}}(\mathbb{R}^n)},
        \end{equation*}
    for all $t>0$, $\varepsilon \in (-\varepsilon_{\omega, p,q},\varepsilon_{\omega, p,q})$ and all $f\in L^2_\omega(\mathbb{R}^n) \cap  L^p_{\omega^{1+(p/q)\varepsilon}}(\mathbb{R}^n)$. In particular, 
    $e^{t\Delta_\omega}$ extends to a bounded operator from $L^p_{\omega^{1+(p/q)\varepsilon}}(\mathbb{R}^n)$ to $L^q_{{\omega}^{q/p+\varepsilon}}(\mathbb{R}^n)$.
    \end{lem}
\begin{proof} The last point follows by density once the inequality is established and we prove it next. 

    Given that $\omega \in RH_{\frac{q}{p}}(\mathbb{R}^n)$, point (4) of Proposition \ref{prop:weights} and Hölder's inequality imply that there exist {$\varepsilon_{\omega, p,q} = \varepsilon([\omega]_{RH_{\frac{q}{p}}},p,q,n) \in \big(0, \min\big(\frac q p, \frac{q}{p'}\big)\big)$} and a constant $C = C([\omega]_{RH_{\frac{q}{p}}},p,q,n) <\infty$ such that, for all $\varepsilon \in (-\varepsilon_{\omega, p,q}, \varepsilon_{\omega, p,q})$ and all cubes $Q \subset \mathbb{R}^n$, the following inequalities hold
    \begin{equation}\label{eq:theta}
        \left ( \frac{1}{|Q|} \int_Q \omega^{\frac{q}{p}+\varepsilon}(x) \ \mathrm{d}x   \right )^{\frac{1}{\frac{q}{p}+\varepsilon}} \leq \frac{C}{|Q|} \int_Q \omega(x) \ \mathrm{d}x,
    \end{equation} 
    \begin{equation}\label{eq:thetatilde}
        \left ( \frac{1}{|Q|} \int_Q \omega^{1-\frac{p'}{q}\varepsilon}(x) \ \mathrm{d}x   \right )^{\frac{1}{1-\frac{p'}{q}\varepsilon} }\leq  \frac{C}{|Q|} \int_Q \omega(x) \ \mathrm{d}x.
    \end{equation}
    We fix $\varepsilon \in (-\varepsilon_{\omega, p,q},\varepsilon_{\omega, p,q})$ and set $\gamma= \frac{1}{p}+\frac{\varepsilon}{q} \in (0,1)$. We fix $f\in L^p_{\omega^{1+(p/q)\varepsilon}}(\mathbb{R}^n) \cap L^2_\omega(\mathbb{R}^n)$ . We fix $t>0$ and let $m \in \mathbb{Z}$ be the integer verifying $\frac{\sqrt{t}}{2}<\frac{1}{2^m}\le \sqrt{t}$. Let $(Q_{k})_{k\in \mathbb{Z}^n}$ the family of dyadic cubes with side lengths $\frac{1}{2^m}$. Using the upper bound for the heat kernel in Lemma \ref{lem:GULB}, we have  
    \begin{align*}
        \| e^{t\Delta_\omega} f &\|^q_{L^q_{{\omega}^{q/p+\varepsilon}}(\mathbb{R}^n)}= \int_{\mathbb{R}^n} |e^{t\Delta_\omega} f(x)|^q\  \omega^{q/p+\varepsilon}(x) \ \mathrm{d}x 
        \\&\leq C \int_{\mathbb{R}^n} \bigg( \int_{\mathbb{R}^n} \frac{e^{-\frac{|x-y|^2}{ct}}}{\omega_t(x)^{\gamma}\omega_t(y)^{1-\gamma}}|f(y)| \ \mathrm{d} \omega(y) \bigg)^q\  \omega^{q/p+\varepsilon}(x) \ \mathrm{d}x 
        \\& \leq C \sum_{k\in \mathbb{Z}^n} \int_{Q_{k}} \bigg( \sum_{\ell \in \mathbb{Z}^n} \int_{Q_{\ell}} \frac{e^{-\frac{|x-y|^2}{ct}}}{\omega_t(y)^{1-\gamma}}|f(y)| \frac{\omega^{\frac{\varepsilon}{q}}(y)}{\omega^{\frac{\varepsilon}{q}}(y)} \ \mathrm{d} \omega(y) \bigg)^q\ \frac{\omega^{q/p+\varepsilon}(x)}{\omega_t(x)^{q\gamma}} \ \mathrm{d}x
        \\& \leq C \sum_{k\in \mathbb{Z}^n} \int_{Q_{k}} \bigg( \sum_{\ell\in \mathbb{Z}^n} \bigg( \int_{Q_{\ell}} \frac{e^{-p'\frac{|x-y|^2}{ct}}}{\omega_t(y)^{p'(1-\gamma)}} \omega(y)^{-\frac{p'}{q}\varepsilon} \ \mathrm{d} \omega(y) \bigg)^{\frac{1}{p'}} \|\omega^{\frac{\varepsilon}{q}}  \, f \|_{L^p_\omega(Q_{\ell})} \bigg)^q \ \frac{\omega^{q/p+\varepsilon}(x)}{\omega_t(x)^{q \gamma}} \ \mathrm{d}x
        \\& \leq C \sum_{k\in \mathbb{Z}^n} \int_{Q_{k}} \bigg( \sum_{\ell \in \mathbb{Z}^n} e^{-\frac{d(Q_{k},Q_{\ell})^2}{ct}} \bigg( \int_{Q_{\ell}} \frac{\omega(y)^{1-\frac{p'}{q}\varepsilon}}{\omega_t(y)^{p'(1-\gamma)}} \ \mathrm{d} y \bigg )^{\frac{1}{p'}} \| f \|_{L^p_{\omega^{1+(p/q)\varepsilon}}(Q_{\ell})} \bigg)^q \ \frac{\omega^{q/p+\varepsilon}(x)}{\omega_t(x)^{q \gamma}} \ \mathrm{d}x 
        \\& = C \sum_{k\in \mathbb{Z}^n}  \bigg( \sum_{\ell \in \mathbb{Z}^n} e^{-\frac{d(Q_{k},Q_{\ell})^2}{ct}} \bigg( \int_{Q_{\ell}} \frac{\omega(y)^{1-\frac{p'}{q}\varepsilon}}{\omega_t(y)^{p'(1-\gamma)}} \ \mathrm{d} y \bigg)^{\frac{1}{p'}} \| f \|_{L^p_{\omega^{1+(p/q)\varepsilon}}(Q_{\ell})} \bigg)^q \int_{Q_{k}} \frac{\omega^{q/p+\varepsilon}(x)}{\omega_t(x)^{q \gamma}} \ \mathrm{d}x.
    \end{align*}
    Using the doubling property \eqref{DoublingMuck} with the equivalence of norms on $\mathbb{R}^n$, we see easily that there exists a constant $C=C(n,D)\ge 1 $ such that 
    $$ \frac{1}{C} \omega(Q_{k}) \leq \omega_t(x) \leq C \omega(Q_{k}), \quad \text{for all} \ k\in \mathbb{Z}^n \ \text{and all} \ x \in Q_{k}.$$
    Furthermore, for all $k,\ell \in \mathbb{Z}^n$ with $k \neq \ell$, we have
    $$d(Q_{k},Q_{\ell}) \ge C(n) \bigg(\frac{|k-\ell|}{2^m}-\frac{1}{2^m}\bigg) \ge C(n)(|k-\ell|-1)\frac{\sqrt{t}}{2}.$$
    Thus, we obtain 
    \begin{equation}\label{eq:smgroupe}
        \| e^{t\Delta_\omega} f \|^q_{L^q_{{\omega}^{q/p+\varepsilon}}(\mathbb{R}^n)} \leq C \sum_{k\in \mathbb{Z}^n}  \bigg( \sum_{\ell\in \mathbb{Z}^n} e^{-\frac{|k-\ell|^2}{c}} \frac{\big( \int_{Q_{\ell}} \omega(y)^{1-\frac{p'}{q}\varepsilon}\ \mathrm{d} y \big)^{\frac{1}{p'}}}{\omega(Q_\ell)^{1-\gamma}}  \| f \|_{L^p_{\omega^{1+(p/q)\varepsilon}}(Q_{\ell})} \bigg)^q \frac{\int_{Q_{k}} \omega^{\frac{q}{p}+\varepsilon}(x) \ \mathrm{d}x}{\omega(Q_k)^{q\gamma}}.
    \end{equation}
    In addition, applying inequality \eqref{eq:thetatilde} yields: 
    \begin{equation}\label{eq:1}
        \left ( \int_{Q_{\ell}} \omega(y)^{1-\frac{p'}{q}\varepsilon}\ \mathrm{d} y \right )^{\frac{1}{p'}} \leq  |Q_\ell|^{\frac{1}{p'}} \left( \frac{\omega(Q_\ell)}{|Q_\ell|} \right)^{\frac{1-\frac{p'}{q}\varepsilon}{p'}}=  |Q_\ell|^{\frac{\varepsilon}{q}} \omega(Q_\ell)^{\frac{1}{p'}-\frac{\varepsilon}{q}}= |Q_\ell|^{\frac{\varepsilon}{q}} \omega(Q_\ell)^{1-\gamma}.
    \end{equation}
    Likewise, since $\frac{q}{p} + \varepsilon = q \gamma$, it follows from \eqref{eq:theta} that:
    \begin{equation}\label{eq:2}
        \frac{\int_{Q_{k}} \omega^{q/p+\varepsilon}(x) \ \mathrm{d}x}{\omega(Q_k)^{q\gamma}} \leq C |Q_k|^{1-(\frac{q}{p}+\varepsilon)}.
    \end{equation}
    Since $|Q_k| = |Q_\ell| \sim t^{n/2}$, combining \eqref{eq:1} with \eqref{eq:2}, inequality \eqref{eq:smgroupe} takes the form:
    \begin{align*}
        \| e^{t\Delta_\omega} f \|^q_{L^q_{{\omega}^{q/p+\varepsilon}}(\mathbb{R}^n)} \leq &C t^{\frac{n}{2} \varepsilon} \times t^{\frac{n}{2}(1-(\frac{q}{p}+\varepsilon))} \sum_{k\in \mathbb{Z}^n}  \bigg ( \sum_{\ell\in \mathbb{Z}^n} e^{-\frac{|k-\ell|^2}{c}}   \| f \|_{L^p_{\omega^{1+(p/q)\varepsilon}}(Q_{\ell})} \bigg )^q 
        \\& = C \times t^{\frac{n}{2}(1-\frac{q}{p})} \sum_{k\in \mathbb{Z}^n}  \bigg ( \sum_{\ell\in \mathbb{Z}^n} e^{-\frac{|k-\ell|^2}{c}}   \| f \|_{L^p_{\omega^{1+(p/q)\varepsilon}}(Q_{\ell})} \bigg )^q. 
    \end{align*}
    Using Young's discrete convolution inequality with $s>1$ such that $\frac{1}{p}+\frac{1}{s}=1+\frac{1}{q}$, we deduce that 
    \begin{align*}
        \| e^{t\Delta_\omega} f \|^q_{L^q_{{\omega}^{q/p+\varepsilon}}(\mathbb{R}^n)} &\leq C t^{\frac{n}{2}(1-\frac{q}{p})} \bigg ( \sum_{k\in \mathbb{Z}^n} e^{-s\frac{|k|^2}{c}} \bigg )^{q/s} \bigg ( \sum_{k\in \mathbb{Z}^n}  \|f \|^p_{L^p_{\omega^{1+(p/q)\varepsilon}}(Q_{k})} \bigg )^{q/p} \\&= C t^{\frac{n}{2}(1-\frac{q}{p})} \bigg( \sum_{k\in \mathbb{Z}^n} e^{-s\frac{|k|^2}{c}} \bigg )^{q/s} \|f \|^q_{L^p_{\omega^{1+(p/q)\varepsilon}}(\mathbb{R}^n)}.
    \end{align*}
     Thus, we obtain
    $$\| e^{t\Delta_\omega} f \|_{L^q_{{\omega}^{q/p+\varepsilon}}(\mathbb{R}^n)} \leq C t^{-\frac{n}{2}(\frac{1}{p}-\frac{1}{q})} \|  f \|_{L^p_{\omega^{1+(p/q)\varepsilon}}(\mathbb{R}^n)}.$$
\end{proof}

When $1<p = q<\infty$, since $\omega \in RH_{s_\omega}(\mathbb{R}^n)$ for some $s_\omega = s([\omega]_{A_2}) > 1$, the proof of Lemma \ref{lem:HLS} still applies with $\varepsilon\ne 0$.

 It is  not clear what happens when $p=1 $ or  $q=\infty$  when $\varepsilon\ne 0$.    In the case $\varepsilon = 0$,  we obtain the following lemma when $1 \leq p < q \leq \infty$ by adapting the above proof.

\begin{lem}\label{lem:extrémités} 
    Let $1 \le p < q \le \infty$. Assume that $\omega \in RH_{\frac{q}{p}}(\mathbb{R}^n)$. Then, there exists a constant $C = C([ \omega  ]_{A_2},[ \omega  ]_{RH_{\frac{q}{p}}},n,p,q)<\infty$ such that 
    \begin{equation*}
        \| e^{t\Delta_\omega} f \|_{L^q_{{\omega}^{q/p}}(\mathbb{R}^n)} \leq C t^{-\frac{n}{2}(\frac{1}{p}-\frac{1}{q})} \|  f \|_{L^p_{\omega}(\mathbb{R}^n)},
    \end{equation*}
    for all $t>0$ and all $f\in L^2_\omega(\mathbb{R}^n) \cap  L^p_{\omega}(\mathbb{R}^n)$. In particular, $e^{t\Delta_\omega}$ extends to a bounded operator from $L^p_{\omega}(\mathbb{R}^n)$ to $L^q_{{\omega}^{q/p}}(\mathbb{R}^n)$. Here, ${L^q_{{\omega}^{q/p}}(\mathbb{R}^n)}$ when $q=\infty$ is interpreted as the space of measurable functions $f$ such that $\omega^{1/p} f \in L^\infty_{\mathrm{d}x}(\mathbb{R}^n)$.
\end{lem}

\subsection{Proof of Theorem \ref{thm:principal}}
    We adapt  \cite[Proposition 5.3]{auscher2007necessary}. By the Calder\'on reproducing formula, we have for all $f\in L^2_\omega(\mathbb{R}^n)$ and $\alpha>0$,  
    \begin{equation*}
                f=\lim_{\substack{\delta \to 0^+ \\ R \to \infty}}\frac{1}{\Gamma(\alpha)}\int_{\delta}^{R}(-t\Delta_\omega)^\alpha e^{t\Delta_\omega}f \, \frac{\mathrm{d}t}{t}=:\frac{1}{\Gamma(\alpha)}\int_{0}^{\infty}(-t\Delta_\omega)^\alpha e^{t\Delta_\omega}f \, \frac{\mathrm{d}t}{t},
    \end{equation*}
    with convergence in $L^2_\omega(\mathbb{R}^n)$. For all $0<\delta<R$, we set $T_{\delta,R}:=\frac{1}{\Gamma(\alpha)}\int_{\delta}^{R}t^{\alpha} e^{t\Delta_\omega} \, \frac{\mathrm{d}t}{t}$, which is defined on $L^2_\omega(\mathbb{R}^n)$, into $L^2_\omega(\mathbb{R}^n)$ with non uniform bounds.
    
    \bigskip

     \noindent\textbf{Step 1.} We first show weak type bounds:  namely,  
     $\omega^{\frac{1}{p}+\frac{\varepsilon}{q}} \, T_{\delta,R}$  extends to a bounded operator from $L^p_{\omega^{1+(p/q)\varepsilon}}(\mathbb{R}^n)$ to the Lorentz space $L^{q,\infty}_{dx}(\mathbb{R}^n)$ with bounds independent of $\delta, R$ for all $1<p<q<\infty$ with $\alpha=\frac{n}{2}(\frac{1}{p}-\frac{1}{q})$ and $\varepsilon$ in a small interval $(-\tilde{\varepsilon}_{\omega, p,q},\tilde{\varepsilon}_{\omega, p,q})$ depending on $[ \omega  ]_{RH_{\frac{q}{p}}}$, $p$, $q$ and $n$. For $\varepsilon=0$, this also holds when $p=1$.
    
    Indeed, fix $(p,q)$ as above and  select $q_0$ and $q_1$ such that
      $p<q_0<q<q_1$  so that
       $\omega \in RH_{\frac{q_0}{p}}(\mathbb{R}^n)\cap RH_{\frac{q_1}{p}}(\mathbb{R}^n)$ (see point (4) of Proposition \ref{prop:weights}). If $|\varepsilon| < \tilde{\varepsilon}_{\omega, p,q}:= q \min \big(\frac{\varepsilon_{\omega, p,q_0}}{q_0},  \frac{\varepsilon_{\omega, p,q_1}}{q_1}\big)$, one can select  $\varepsilon_0, \varepsilon_1$ with $\frac{\varepsilon_0}{q_0}=\frac{\varepsilon}{q}=\frac{\varepsilon_1}{q_1}$,  so that  Lemma \ref{lem:HLS} applies to each triplet $(p,q_0, \varepsilon_0)$ and $(p,q_1, \varepsilon_1)$.   Hence,  Minkowski inequality with the relation between $p,q$ and $\alpha$ imply that for  any $\delta <a<R    $
\begin{equation*}
      \| T_{\delta,a}f\|_{L^{q_0}_{\omega^{q_0/p+\varepsilon_0}}(\mathbb{R}^n)} \le \frac{1}{\Gamma(\alpha)}   \int_{\delta }^{a} t^{\alpha -1} \|e^{t\Delta_\omega}f\|_{L^{q_0}_{\omega^{q_0/p+\varepsilon_0}}(\mathbb{R}^n)} \, \mathrm dt  \leq C a^{\frac{n}{2}(\frac{1}{q_0}-\frac{1}{q})} 
      \|  f \|_{L^p_{\omega^{1+(p/q_0)\varepsilon_0}}(\mathbb{R}^n)},
    \end{equation*}
    and 
    \begin{equation*}
       \| T_{a,R}f\|_{L^{q_1}_{\omega^{q_1/p+{ \varepsilon_1}}}(\mathbb{R}^n)} \le \frac{1}{\Gamma(\alpha)}  \int_{a}^{R} t^{\alpha -1} \|e^{t\Delta_\omega}f\|_{L^{q_1}_{\omega^{q_1/p+{\varepsilon_1}}}(\mathbb{R}^n)} \, \mathrm dt  \leq C a^{\frac{n}{2}(\frac{1}{q_1}-\frac{1}{q})} 
       \|  f \|_{L^p_{\omega^{1+(p/q_1){\varepsilon_1}}}(\mathbb{R}^n)}.
    \end{equation*}
     We fix $f\in L^p_{\omega^{1+(p/q)\varepsilon}}(\mathbb{R}^n) \cap L^2_\omega(\mathbb{R}^n)$ such that $\|  f \|_{L^p_{\omega^{1+(p/q)\varepsilon}}(\mathbb{R}^n)}=1$. As $T_{\delta,R}=T_{\delta,a}+T_{a,R}$, then for all $\lambda>0$, Markov inequality for the Lebesgue measure implies that
    \begin{align*}
        |\{\omega^{\frac{1}{p}+\frac{\varepsilon}{q}} \, |T_{\delta,R}f | >\lambda  \} | &\leq
        |\{\omega^{\frac{1}{p}+\frac{\varepsilon_0}{q_0}} \, |T_{\delta,a}f | >\frac{\lambda }{2} \} |+ |\{\omega^{\frac{1}{p}+\frac{\varepsilon_1}{q_1}} \, |T_{a,R}f | >\frac{\lambda }{2} \} |
        \\&\leq C \lambda^{-q_0}a^{\frac{n q_0}{2}(\frac{1}{q_0}-\frac{1}{q})}+ C \lambda^{-q_1}a^{\frac{n  q_1}{2}(\frac{1}{q_1}-\frac{1}{q})}.
    \end{align*}
    Taking $a=\lambda^{-\frac{2q}{n}}$, we obtain the unweighted weak type bound 
    \begin{equation}\label{eq:weak}
        |\{\omega^{\frac{1}{p}+\frac{\varepsilon}{q}} \, |T_{\delta,R}f| >\lambda  \} | \leq 2C \, \lambda^{-q}.
    \end{equation}
    Remark that if $\lambda^{-\frac{2q}{n}}\le\delta$ or $\lambda^{-\frac{2q}{n}}\ge R$, then we obtain the same estimate by using one of the two terms. 
    
    According to Lemma \ref{lem:extrémités}, the weak-type estimate \eqref{eq:weak} remains valid in the endpoint case $p = 1$ with $\varepsilon = 0$, that is 
    \begin{equation}\label{eq:weak11}
        |\{\omega \, |T_{\delta,R}f| >\lambda  \} | \leq 2C \, \lambda^{-q}.
    \end{equation}
    
    \bigskip
    
    \noindent\textbf{Step 2.} The second step, relevant only for $p > 1$, consists in interpolating the above weak type inequalities obtained for different triplets $(p_i, q_i, \varepsilon_i)$ as follows. We fix again $(p,q)$ with $\alpha=\frac{n}{2}(\frac{1}{p}-\frac{1}{q})$. We may choose $1<p_0<p<p_1$, $q_0<q<q_1<\infty$ and $|\varepsilon_0|<\tilde{\varepsilon}_{\omega, p_0,q_0}, |\varepsilon_1|<\tilde{\varepsilon}_{\omega, p_1,q_1}$  such that $$\omega \in RH_{\frac{{q}_1}{p_0}}(\mathbb{R}^n) \subset RH_{\frac{{q}_1}{p_1}}(\mathbb{R}^n) \cap RH_{\frac{{q}_0}{p_0}}(\mathbb{R}^n),$$ 
    $$ \frac{1}{p_0}-\frac{1}{q_0}=\frac{1}{p}-\frac{1}{q}=\frac{1}{p_1}-\frac{1}{q_1},$$
    and 
    $$ \frac{1}{p_0}+\frac{\varepsilon_0}{q_0}=\frac{1}{p} = \frac{1}{p_1}+\frac{\varepsilon_1}{{q}_1} .$$
    Note that $\varepsilon_0$ and $\varepsilon_1$ must have different signs. 
    By the first step, we have
    \begin{equation}\label{eq:weak1}
       \omega^{\frac 1 p}\, T_{\delta,R}= \omega^{\frac{1}{p_0}+\frac{\varepsilon_0}{{q}_0}} \, T_{\delta,R}:  
       L^{p_0}_{\omega^{1+(p_0/{q}_0)\varepsilon_0 }}(\mathbb{R}^n)=L^{p_0}_{\omega^{p_0/p }}(\mathbb{R}^n)  \longrightarrow L^{q_0,\infty}_{\mathrm{d}x}(\mathbb{R}^n),
    \end{equation}
    \begin{equation}\label{eq:weak2}
       \omega^{\frac 1 p}\, T_{\delta,R}=\omega^{\frac{1}{p_1}+\frac{\varepsilon_1}{{q}_1}} \, T_{\delta,R}:   
       L^{p_1}_{\omega^{1+(p_1/{q}_1)\varepsilon_1 }}(\mathbb{R}^n)=L^{p_1}_{\omega^{p_1/p }}(\mathbb{R}^n) \longrightarrow L^{q_1,\infty}_{\mathrm{d}x}(\mathbb{R}^n).
    \end{equation}
    Since ${\frac{1}{p}=\frac{1-\theta}{p_0}+\frac{\theta}{p_1}}$ and ${\frac{1}{q}=\frac{1-\theta}{q_0}+\frac{\theta}{q_1}} $ for some $\theta \in (0,1)$, 
    by Stein-Weiss real interpolation theorem with change of measures \cite[Theorem 2.9]{stein1958interpolation}, 
    $$\omega^{\frac 1 p}\, T_{\delta,R}:   L^{p}_{\tilde\omega}(\mathbb{R}^n) \longrightarrow L^{q}_{\mathrm{d}x}(\mathbb{R}^n)$$
    with 
    $$\Tilde{\omega}= \left ( \omega^{\frac{p_0}p} \right )^{(1-\theta)\frac{p}{p_0}} \times \left ( \omega^{\frac{p_1}p} \right )^{\theta \frac{p}{p_1}}=\omega. $$
     Thus, we have proved that
    \begin{equation}\label{eq:TepsR}
         \| T_{\delta,R}f\|_{L^{q}_{\omega^{q/p}}(\mathbb{R}^n)} = \| \omega^{\frac 1 p}T_{\delta,R}f\|_{L^{q}(\mathbb{R}^n)}\leq C  \|f \|_{L^p_\omega(\mathbb{R}^n)}, \quad \text{for all} \ f \in  L^2_\omega(\mathbb{R}^n)\cap L^p_\omega(\mathbb{R}^n),
    \end{equation}
    with $C$ independent of $f$, $\delta$ and $R$.
    
    \bigskip
    
    \noindent\textbf{Step 3.} {The third step consists in passing to the limit to obtain \eqref{eq:fractionalHLS} and \eqref{eq:fractionalHLSweak}. We begin with the case $p>1$.   The inequality \eqref{eq:TepsR} remains valid for all $f \in L^2_\omega(\mathbb{R}^n)$ (as the inequality is obvious if $\|f \|_{L^p_\omega(\mathbb{R}^n)}=\infty$). Hence, substituting $f$ by $(-\Delta_\omega)^\alpha f\in L^2_\omega(\mathbb{R}^n)$ for   $f\in D((-\Delta_\omega)^\alpha)$, we have 
    $$\left\| T_{\delta,R}(-\Delta_\omega)^\alpha f\right\|_{L^{q}_{\omega^{q/p}}(\mathbb{R}^n)} \leq C  \|(-\Delta_\omega)^\alpha f \|_{L^p_\omega(\mathbb{R}^n)}. $$
    As $ T_{\delta,R}(-\Delta_\omega)^\alpha f$ tends to $f$ in $L^2_\omega(\mathbb{R}^n)$ when $\delta \to 0^+$ and $R \to \infty$ hence almost everywhere in $\mathbb{R}^n$ up to extracting a sub-sequence, then by Fatou's lemma,
    \begin{equation}\label{eq:CRF}
        \left\|  f\right\|_{L^{q}_{\omega^{q/p}}(\mathbb{R}^n)} \leq \liminf_{\substack{\delta \to 0^+ \\ R \to \infty}} \, \left\| T_{\delta,R}(-\Delta_\omega)^\alpha f\right\|_{L^{q}_{\omega^{q/p}}(\mathbb{R}^n)} \leq C \|(-\Delta_\omega)^\alpha f \|_{L^p_\omega(\mathbb{R}^n)}.
    \end{equation}
    {When $p=1$, using  the weak-type estimate \eqref{eq:weak11} and convergence in measure, we obtain 
    \begin{equation*}
        \| \omega  f \|_{L^{q,\infty}_{\mathrm{d}x}(\mathbb{R}^n)} \leq C \| (-\Delta_\omega)^\alpha f \|_{L^1_{\omega}(\mathbb{R}^n)}, \quad \forall f \in D((-\Delta_\omega)^\alpha).
    \end{equation*}}
    
    \bigskip
    
    \noindent\textbf{Step 4.} To obtain the inequalities for $(-\Delta_\omega)^{-\alpha}$, one can observe that  $T_{\delta,R}f \to (-\Delta_\omega)^{-\alpha} f  $ in $L^2_\omega(\mathbb{R}^n)$ when $f\in D((-\Delta_\omega)^{-\alpha})$. Hence by the same  arguments starting either from \eqref{eq:TepsR} or \eqref{eq:weak11}, we obtain the desired inequalities on  $D((-\Delta_\omega)^{-\alpha}) \cap L^p_\omega(\mathbb{R}^n)$. It remains to prove that, for all $1\le p <\infty$,  $D((-\Delta_\omega)^{-\alpha}) \cap L^p_\omega(\mathbb{R}^n)$ is dense in $L^p_\omega(\mathbb{R}^n)$. Since the domains form an increasing sequence of spaces in $-\alpha$, it suffices to prove it when  $\alpha$ is an integer, say $N\in \mathbb{N}$, $N\ge 1$. We define $\Phi(z) = c z^N e^{-z}$, where 
    $c > 0$ is chosen so that
$$
\int_0^{\infty}  \Phi(t) \, \frac{\mathrm{d}t}{t} = 1.
$$
For every $\ell\in  \mathbb{N}^*$ and $z \in \mathbb{C}$, we define
$$
\varphi_\ell(z) := \int_{\frac{1}{\ell}}^\ell  \Phi(tz) \, \frac{\mathrm{d}t}{t}.
$$
Fix $f \in L^p_\omega(\mathbb{R}^n) \cap L^2_\omega(\mathbb{R}^n)$. Clearly it follows that for all $\ell\in  \mathbb{N}$, $\ell\ge 1$,
$$
\varphi_\ell(-\Delta_\omega) f \in D((-\Delta_\omega)^{-N}) \cap L^p_\omega(\mathbb{R}^n).
$$
Using integration by parts, one can verify by induction on $N$ that 
$$\varphi_\ell(z)=-P_N(-\ell z)e^{-\ell z}+ P_N(-\ell^{-1}z)e^{-\ell^{-1} z}$$ where $P_N$  is a polynomial of degree $N-1$ with $P_N(0)=1$. Therefore,   $\varphi_\ell(-\Delta_\omega)=-P_N (\ell \Delta_\omega) e^{\ell \Delta_\omega}+P_N (\ell^{-1} \Delta_\omega) e^{{\ell}^{-1} \Delta_\omega}$.
By Proposition \ref{prop:boundednessandconvergence}, we have $\varphi_\ell(-\Delta_\omega) f \to f$  in $L^p_\omega(\mathbb{R}^n)$ as $\ell \to \infty$. Thus, $D((-\Delta_\omega)^{-\alpha}) \cap L^p_\omega(\mathbb{R}^n)$ is dense in $L^2_\omega(\mathbb{R}^n) \cap L^p_\omega(\mathbb{R}^n)$, and hence dense in $L^p_\omega(\mathbb{R}^n)$. \hspace{2.5cm} $\square$

\begin{rem}
    Having the freedom of choosing small $\varepsilon$ from Lemma \ref{lem:HLS} is what allows us to have the same operator in \eqref{eq:weak1} and \eqref{eq:weak2}. Otherwise, we could not interpolate using  the Stein-Weiss theorem.
\end{rem}

\subsection{Extension to weighted Lorentz spaces}\label{sub:Lorentz}
Let $\mu$ be a measure on $\mathbb{R}^n$. The Lorentz space $L^{p,r}_{\mu}(\mathbb{R}^n)$ is defined as the set of all measurable functions $f$ on $\mathbb{R}^n$ such that $\left\| f\right\|_{L^{p,r}_{\mu}(\mathbb{R}^n)} < \infty$, with
\begin{align*}
\left\| f\right\|_{L^{p,r}_{\mu}(\mathbb{R}^n)} =\left\{\begin{matrix}
\left (r \int_{0}^{\infty} s^{r-1} \lambda_f(s)^{r/p} \, \mathrm{d}s \right )^{1/r}, \quad \ 1\leq p,r <\infty , \\ 
 \sup_{s>0} s \, \lambda_f(s)^{1/p}, \quad \ 1\le p < \infty, \ r=\infty,
 \\ 
 \end{matrix}\right. 
 \end{align*}
 where 
$$ \lambda_f(s)=\mu\left ( \left\{ x\in \mathbb{R}^n : |f(x)|>s\right\} \right )
. 
$$
It is known that $L^{p,p}_{\mu}(\mathbb{R}^n)=L^{p}_{\mu}(\mathbb{R}^n)$ with $\left\| f\right\|_{L^{p,p}_{\mu}(\mathbb{R}^n)}=\left\| f\right\|_{L^{p}_{\mu}(\mathbb{R}^n)}$ and $\left\| f\right\|_{L^{p,r}_{\mu}(\mathbb{R}^n)}$ is non-increasing as a function of $r$, so that $L^{p,r_2}_{\mu}(\mathbb{R}^n)
\subset L^{p,r_1}_{\mu}(\mathbb{R}^n)$ if $r_1 \leq r_2$. We refer to \cite[Ch. V, \S3 ]{stein1971introduction} for more details. If $\Tilde{\omega}$ is a weight on $\mathbb{R}^n$, then $L^{p,r}_{\mu}(\Tilde{\omega}, \mathbb{R}^n)$ is the space of all measurable functions $f$ on $\mathbb{R}^n$ such that $\Tilde{\omega}f$ belongs to the Lorentz space $L^{p,r}_{\mu}(\mathbb{R}^n)$ and we set 
\begin{equation*}
  \left\| f \right\|_{L^{p,r}_{\mu}(\Tilde{\omega}, \mathbb{R}^n)} :=  \left\| \Tilde{\omega}f \right\|_{L^{p,r}_{\mu}(\mathbb{R}^n)}.
\end{equation*}

Theorem \ref{thm:principal}  can be extended to obtain boundedness into Lorentz spaces. There are two possible results, and it is not clear that they are connected.   

\begin{thm}\label{thm:principal*}
   Let $1<p<q<\infty$. Assume that $\omega \in RH_{\frac{q}{p}}(\mathbb{R}^n)$ and set $\alpha=\frac{n}{2}(\frac{1}{p}-\frac{1}{q})$. Then, there exists a constant $C = C([ \omega  ]_{A_2},[ \omega  ]_{RH_{\frac{q}{p}}},n,p,q)<\infty$ such that
    \begin{equation*}
        \| f \|_{L^{q,p}_{\omega}(\omega^{\frac{1}{p}-\frac{1}{q}}, \mathbb{R}^n)} \leq C \| (-\Delta_\omega)^\alpha f \|_{L^p_{\omega}(\mathbb{R}^n)}, \quad \forall f \in D((-\Delta_\omega)^\alpha).
    \end{equation*}
    In particular, $(-\Delta_\omega)^{-\alpha}$ defines a bounded operator from $L^p_\omega(\mathbb{R}^n)$ to $L^{q,p}_{\omega}(\omega^{\frac{1}{p}-\frac{1}{q}}, \mathbb{R}^n)$.
\end{thm}
\begin{proof}
    We resume the proof of Theorem \ref{thm:principal} from equation \eqref{eq:TepsR}, where we have proved that $T_{\delta,R}$ extends to a bounded operator from $L^p_\omega(\mathbb{R}^n)$ to $ L^q_{\omega^{q/p}}(\mathbb{R}^n)$. Now, to interpolate, we fix exponents $1<p_0<p<p_1$ and $q_0<q<q_1$ such that $\omega \in RH_{\frac{q_1}{p_0}}(\mathbb{R}^n)$ and 
    $$ \frac{1}{p_0}-\frac{1}{q_0}=\frac{1}{p}-\frac{1}{q}=\frac{1}{p_1}-\frac{1}{q_1}.$$
    One can thus choose $\theta \in (0,1)$  such that $\frac{1}{p}=\frac{1-\theta}{p_0}+\frac{\theta}{p_1}$ and $\frac{1}{q}=\frac{1-\theta}{q_0}+\frac{\theta}{q_1}$.
    Using the real interpolation method, for all $1\le r \le \infty$, we have 
    \begin{equation*}
        \left ( L^{p_0}_{\omega}(\mathbb{R}^n), L^{p_1}_{\omega}(\mathbb{R}^n) \right )_{\theta,r}=L^{p,r}_\omega(\mathbb{R}^n) \quad \text{and} \quad  ( L^{q_0}_{\omega^{q_0/p_0}}(\mathbb{R}^n), L^{q_1}_{\omega^{q_1/p_1}}(\mathbb{R}^n) )_{\theta,r}=L^{q,r}_\omega(\omega^{\frac{1}{p}-\frac{1}{q}},\mathbb{R}^n),
    \end{equation*}
     The first interpolation result, without a change of measure, is well known Hunt interpolation theorem and follows as a special case of \cite[Theorem 5.3.1]{bergh2012interpolation}. The second, involving a change of measure, is known as the Lizorkin–Freitag formula and is due to Lizorkin and Freitag (see \cite[Theorem 2]{Freitag1978} and \cite{Lizorkin1975}). Thus, for all $1\le r \le \infty$, $T_{\delta,R}: L^{p,r}_\omega(\mathbb{R}^n) \rightarrow L^{q,r}_\omega(\omega^{\frac{1}{p}-\frac{1}{q}},\mathbb{R}^n)$ is bounded. We take $r=p$, and conclude the argument by following the same final steps as in the proof of Theorem \ref{thm:principal}.
\end{proof}
\begin{rem}
    Theorem \ref{thm:principal*} improves Theorem \ref{thm:principal}, since $p<q$, so for every $f \in L^{q,p}_{\omega}(\omega^{\frac{1}{p}-\frac{1}{q}}, \mathbb{R}^n)$:
    \begin{equation*}
        \| f \|_{L^q_{{\omega}^{q/p}}(\mathbb{R}^n)}=\| f \|_{L^{q,q}_{\omega}(\omega^{\frac{1}{p}-\frac{1}{q}}, \mathbb{R}^n)} \le  \| f \|_{L^{q,p}_{\omega}(\omega^{\frac{1}{p}-\frac{1}{q}}, \mathbb{R}^n)}.
    \end{equation*}
\end{rem}

The second result, still improving Theorem \ref{thm:principal},  is the following.     

\begin{thm}
\label{cor:principal**}
   Let $1< p<q<\infty$. Assume that $\omega \in RH_{\frac{q}{p}}(\mathbb{R}^n)$ and set $\alpha=\frac{n}{2}(\frac{1}{p}-\frac{1}{q})$. Then, there exists a constant $C = C([ \omega  ]_{A_2},[ \omega  ]_{RH_{\frac{q}{p}}},n,p,q)<\infty$ such that
    \begin{equation*}
        \| \omega^{\frac{1}{p}}  f \|_{L^{q,p}_{\mathrm{d}x}(\mathbb{R}^n)} \leq C \| (-\Delta_\omega)^\alpha f \|_{L^p_{\omega}(\mathbb{R}^n)}, \quad \forall f \in D((-\Delta_\omega)^\alpha).
    \end{equation*}
    In particular, $\omega^{\frac{1}{p}} \, (-\Delta_\omega)^{-\alpha}$ defines a bounded operator from $L^p_\omega(\mathbb{R}^n)$ to $L^{q,p}_{\mathrm{d}x}(\mathbb{R}^n)$.
\end{thm}

\begin{proof} 
The proof can also be seen by resuming the one of Theorem \ref{thm:principal} from equations \eqref{eq:weak1} and \eqref{eq:weak2} with the same notation, where it was shown that the operator $\omega^{\frac{1}{p}} \, T_{\delta,R}$ extends boundedly from $L^{p_0}_{\omega^{p_0/p}}(\mathbb{R}^n)$ to $L^{q_0,\infty}_{\mathrm{d}x}(\mathbb{R}^n)$ and from $L^{p_1}_{\omega^{p_1/p}}(\mathbb{R}^n)$ to $L^{q_1,\infty}_{\mathrm{d}x}(\mathbb{R}^n)$. With the definition of these exponents, one can choose $\theta \in (0,1)$  such that $\frac{1}{p} = \frac{1-\theta}{p_0} + \frac{\theta}{p_1}$ and $\frac{1}{q} = \frac{1-\theta}{q_0} + \frac{\theta}{q_1}$.
    By applying the real interpolation method, we know that
    \begin{equation*}
    \left( L^{p_0}_{\omega^{p_0/p}}(\mathbb{R}^n), L^{p_1}_{\omega^{p_1/p}}(\mathbb{R}^n) \right)_{\theta, p} = L^{p}_\omega(\mathbb{R}^n) \quad \text{and} \quad
    \left( L^{q_0,\infty}_{\mathrm{d}x}(\mathbb{R}^n), L^{q_1,\infty}_{\mathrm{d}x}(\mathbb{R}^n) \right)_{\theta, p} = L^{q,p}_{\mathrm{d}x}(\mathbb{R}^n).
    \end{equation*} The first interpolation identity, which involves a change of measure, is due to Peetre; see \cite[Theorem 5.5.1]{bergh2012interpolation}. The second identity is, again, Hunt's interpolation theorem; see \cite[Theorem 5.3.1]{bergh2012interpolation}. Hence, the operator $\omega^{\frac{1}{p}} \, T_{\delta,R}$ is bounded from $L^{p}_\omega(\mathbb{R}^n)$ to $L^{q,p}_{\mathrm{d}x}(\mathbb{R}^n)$. The conclusion  is then obtained by following the same final steps as in the proof of Theorem \ref{thm:principal}.
\end{proof}

\begin{rem}\label{lem:Lorentzs} The second result would be  a corollary of the first one if one could prove that 
     the multiplication operator
    \begin{align*}
        M_{\omega^{1/p}} : L^{q,p}_{\omega}(\omega^{\frac{1}{p} - \frac{1}{q}}, \mathbb{R}^n) 
        \longrightarrow L^{q,p}_{{\mathrm{d}x}}(\mathbb{R}^n), \quad 
        f \longmapsto \omega^{1/p} f,
    \end{align*}
    defines a bounded operator when  $1 < p< q<\infty$ and $\omega \in RH_{\frac{q}{p}}(\mathbb{R}^n)$. It is  trivially true when $p=q$. One can show the desired inequality when $f$ is an indicator function of a measurable set $A \subset \mathbb{R}^n$, but we have been unable to prove it for general functions or to find a counter-example to the statement. In \cite{MR4440073}, it is explicitly said in the introduction that while considering $\omega$ as part of the measure or as a multiplicative function yield the same Lebesgue space,  it does not necessarily lead to the same Lorentz space. They took the example of $p=\infty$.  Here we take $1<p<q$ and we are ready to assume the reverse H\"older condition on $\omega$, but it does not seem to suffice for a conclusion. Hence, we obtain these two different results and neither one is more natural than the other.
\end{rem}

\begin{rem}
    As in the unweighted case with Hunt's interpolation theorem, one might expect that starting from the weak-type estimates \eqref{eq:weak1} and \eqref{eq:weak2}, one could deduce that
    \begin{equation*}
        \omega^{\frac{1}{p}} \, T_{\delta,R} : L^{p,r}_\omega(\mathbb{R}^n) \longrightarrow L^{q,r}_{\mathrm{d}x}(\mathbb{R}^n), \quad \text{for all} \ 1\le r \le \infty.
    \end{equation*}
    However, in general, this holds only for $r = p$, as illustrated by a counterexample in \cite{Ferreyra1997}. In fact, by applying the Lizorkin--Freitag formula \cite[Theorem 2]{Freitag1978}, we actually obtain that
     $$ 
     \omega^{\frac{1}{p}} \, T_{\delta,R} : L^{p,r}_{\mathrm{d}x}\big(\omega^{1/p},\mathbb{R}^n\big) \longrightarrow L^{q,r}_{\mathrm{d}x}(\mathbb{R}^n), \quad \text{for all} \ 1\le r \le \infty.$$ 
    The case $r = p$ corresponds to what has already been used in the proof of Theorem \ref{cor:principal**}. But when $r\ne p$, the two source spaces in the displayed formulas are not the same, in agreement with both the counter-example and the above remark.
    \end{rem}

\subsection{Extension to the functional calculus}\label{sec:bhfc}

We present here a consequence of Theorem \ref{thm:principal}.

We recall that $\varphi \in \mathcal{H}^\infty_0(S_\mu)$, for $\mu \in (0, \pi)$ where $S_{\mu} := \left\{ z \in \mathbb{C} \setminus \{0\} , : |\mathrm{arg}(z)| < \mu \right\}$, if $\varphi: S_{\mu} \rightarrow \mathbb{C}$ is holomorphic and there exist constants $C, s > 0$ such that
\begin{equation*}
    |\varphi(z)| \leq C \min(|z|^s, |z|^{-s}), \quad \text{for all} \ z \in S_\mu.
\end{equation*}
Let $p \in (1, \infty)$. We say that $-\Delta_\omega$ admits a bounded holomorphic functional calculus on $L^p_\omega(\mathbb{R}^n)$ if, for every $\mu \in (0, \pi)$ and every $\varphi \in \mathcal{H}^\infty_0(S_\mu)$, and for all $f \in L^2_\omega(\mathbb{R}^n) \cap L^p_\omega(\mathbb{R}^n)$, the following estimate holds:
\begin{equation}\label{eq:Hinfini}
    \| \varphi(-\Delta_\omega)f \|_{L^p_\omega(\mathbb{R}^n)} \leq C \| \varphi \|_{\infty,\mu} \| f \|_{L^p_\omega(\mathbb{R}^n)},
\end{equation}
where $C > 0$ is a constant independent of both $\varphi$ and $f$, and where $\| \varphi \|_{\infty,\mu}$ denotes the $L^\infty$-norm of $\varphi$ on the sector $S_\mu$. By a standard density argument, inequality \eqref{eq:Hinfini} implies that $\varphi(-\Delta_\omega)$ extends to a bounded operator on $L^p_\omega(\mathbb{R}^n)$. Moreover, this estimate remains valid for any bounded holomorphic function $\varphi$ on $S_\mu$. This extension follows readily by combining Fatou's lemma with McIntosh’s convergence lemma in $L^2_\omega(\mathbb{R}^n)$ (see \cite{McIntosh86}).

By combining the estimates in \eqref{eq:GULBComplexe} with the fact that $-\Delta_\omega$ clearly has a bounded $H^\infty$-calculus on $L^2_\omega(\mathbb{R}^n)$, one can apply a result due to Duong and Robinson \cite[Theorem 3.1]{duong1996semigroup} concerning the boundedness of the $H^\infty$-calculus of $-\Delta_\omega$ on $L^p_\omega(\mathbb{R}^n)$.
\begin{lem}\label{lem:DR}
    The operator $-\Delta_\omega$ admits a bounded $H^\infty$-calculus on $L^p_\omega(\mathbb{R}^n)$ for each $p \in (1,\infty)$.
\end{lem}

We made a point of not using this property so far, using only elementary arguments. Incorporating this property, we can enhance our results as follows.

\begin{cor} \label{cor:BHFC2}
Let $n\ge 1$ and $\omega \in A_2(\mathbb{R}^n)$.  Let $1<p<q<\infty$. Assume that $\omega \in RH_{\frac{q}{p}}(\mathbb{R}^n)$ and set $\alpha=\frac{n}{2}(\frac{1}{p}-\frac{1}{q})$. Assume that $\varphi$ is a holomorphic function on $S_{\mu}$ and let $\psi(z)= z^{-\alpha}\varphi(z)$.
\begin{enumerate}
    \item Assume $\psi$ is bounded. Then there exists a constant $C = C([ \omega  ]_{A_2},[ \omega  ]_{RH_{\frac{q}{p}}},n,p,q)<\infty$ such that
    \begin{equation*}
       \| \omega^{1/p}\psi(-\Delta_\omega) f \|_{L^q_{\mathrm{d}x}(\mathbb{R}^n)} \leq C \| \varphi(-\Delta_\omega) f \|_{L^p_{\omega}(\mathbb{R}^n)}, \quad \forall f \in D((-\Delta_\omega)^\alpha).
    \end{equation*}
    \item Assume $\varphi$ is bounded. Then $\psi(-\Delta_\omega)$ defines a bounded operator from $L^p_\omega(\mathbb{R}^n)$ to $L^q_{{\omega}^{q/p}}(\mathbb{R}^n)$ or, equivalently, $w^{1/p}\psi(-\Delta_\omega)$ is bounded from $L^p_\omega(\mathbb{R}^n)$ to $L^q_{\mathrm{d}x}(\mathbb{R}^n)$, 
with bound $C \|\varphi\|_{\infty, \mu}$ with $C$ a constant as above also depending on $\mu$. 
\end{enumerate}
\end{cor}

\begin{proof}
    
    {For the first point, we employ the standard argument (see \cite[Theorem 5.15]{egert2024harmonic}) stating that if $f \in D((-\Delta_\omega)^\alpha)$ and $\psi$ is bounded, then $g = \psi(-\Delta_\omega)f \in D((-\Delta_\omega)^\alpha)$, with $(-\Delta_\omega)^\alpha g = \psi(-\Delta_\omega)(-\Delta_\omega)^\alpha f = \varphi(-\Delta_\omega)f$. Applying inequality \eqref{eq:fractionalHLS} to $g$ and substituting completes the argument.}
    
    For the second point, let $f\in L^p_{\omega}(\mathbb{R}^n)$ and set $g=\varphi(-\Delta_\omega)f\in L^p_{\omega}(\mathbb{R}^n)$ by Lemma \ref{lem:DR}. Then using \eqref{eq:fractionalHLS} and \eqref{eq:Hinfini}, 
    \begin{equation*}
       \| \omega^{1/p}\psi(-\Delta_\omega) f \|_{L^q_{\mathrm{d}x}(\mathbb{R}^n)}=\| \omega^{1/p}(-\Delta_\omega)^{-\alpha} g \|_{L^q_{\mathrm{d}x}(\mathbb{R}^n)}  \leq C \| g \|_{L^p_{\omega}(\mathbb{R}^n)} \le C' \| \varphi \|_{\infty,\mu} \|  f \|_{L^p_{\omega}(\mathbb{R}^n)}.
    \end{equation*}
\end{proof}

\begin{rem}
    By the same interpolation argument as in the proof of Theorem \ref{thm:principal*}, when $\varphi$ is bounded, the operator $\psi(-\Delta_\omega)$ is bounded from $L^p_\omega(\mathbb{R}^n)$ to $L^{q,p}_\omega\big(\omega^{\frac{1}{p}-\frac{1}{q}}, \mathbb{R}^n\big)$, with a bound of the form $C \|\varphi\|_{\infty, \mu}$, where $C$ is a positive constant depending only on $[\omega]_{A_2}$, $[\omega]_{RH_{\frac{q}{p}}}$, $n$, $p$, $q$, and $\mu$.
\end{rem}

\subsection{Sharpness of the condition on the weight}
In this subsection, we check that the H\"older inverse condition assumed to $\omega$ in Theorem \ref{thm:principal} and Lemma \ref{lem:HLS} is the optimal condition.
\begin{prop}
    Let $1 \le p<q \le \infty$ and set $\alpha=\frac{n}{2}(\frac{1}{p}-\frac{1}{q})$. Consider the following two assertions:
    \begin{enumerate}
        \item There is a constant $C <\infty $ such that for all $t > 0$ and $f \in L^2_\omega(\mathbb{R}^n) \cap L^p_\omega(\mathbb{R}^n)$, the following inequality holds:
    \begin{equation*}
        \| e^{t\Delta_\omega} f \|_{L^q_{{\omega}^{q/p}}(\mathbb{R}^n)} \leq C t^{-\alpha} \|  f \|_{L^p_{\omega}(\mathbb{R}^n)}.
    \end{equation*}
    \item Assume $1<p$ and $q<\infty$. There is a constant $C < \infty$ such that for all $f \in D((-\Delta_\omega)^\alpha)$, the following inequality holds:
    \begin{equation*}
        \| f \|_{L^q_{{\omega}^{q/p}}(\mathbb{R}^n)} \leq C \| (-\Delta_\omega)^\alpha f \|_{L^p_{\omega}(\mathbb{R}^n)}.
    \end{equation*}
    \end{enumerate}
    If either (1) or (2) holds, then $\omega \in RH_{\frac{q}{p}}(\mathbb{R}^n)$. 
\end{prop}
\begin{proof}
    Assume that (1) holds. Consider the case $q<\infty$. Let $Q \subset \mathbb{R}^n$ a cube with length $\ell >0$. By assumption, we have 
    $$\| e^{\ell^2\Delta_\omega} \mathbb{1}_Q \|_{L^q_{{\omega}^{q/p}}(\mathbb{R}^n)} \leq C \ell^{-n(\frac{1}{p}-\frac{1}{q})} \|  \mathbb{1}_Q \|_{L^p_{\omega}(\mathbb{R}^n)}. $$
    In particular, 
    $$ \| e^{\ell^2\Delta_\omega} \mathbb{1}_Q \|^q_{L^q_{{\omega}^{q/p}}(Q)} \leq C \ell^{-n(\frac{q}{p}-1)} \omega(Q)^{q/p} = C |Q|^{1-\frac{q}{p}} \omega(Q)^{q/p}.$$
    By using the lower bound on the heat kernel in Lemma \ref{lem:GULB}, we have
    $$ \int_{Q} \bigg( \int_Q \frac{e^{-\frac{c|x-y|^2}{\ell^2}}}{\omega_{\ell^2}(y)} \, \mathrm d \omega(y)  \bigg)^q \omega^{q/p}(x) \, \mathrm d x \leq C |Q|^{1-\frac{q}{p}} \omega(Q)^{q/p}. $$
    Using the doubling property \eqref{DoublingMuck} with the equivalence of norms on $\mathbb{R}^n$, we see easily that there exits a constant $\gamma=\gamma(n,D)>0$ such that 
    $$ \gamma \leq \inf_{x\in Q} \int_Q \frac{e^{-\frac{c|x-y|^2}{\ell^2}}}{\omega_{\ell^2}(y)} \, \mathrm d \omega(y) $$
    Thus, $$  \int_{Q} \omega^{q/p}(x) \, \mathrm d x \leq C |Q|^{1-\frac{q}{p}} \omega(Q)^{q/p}. $$
    Therefore, $\omega \in RH_{\frac{q}{p}}(\mathbb{R}^n)$. 
    
    When $q=\infty$ we recall that (1) reads $\| w^{1/p}e^{t\Delta_\omega} f \|_{L^\infty(\mathbb{R}^n)} \leq C t^{-\alpha} \|  f \|_{L^p_{\omega}(\mathbb{R}^n)}$ and the argument is similar using the lower bound.

    Now, assume that (2) holds. Let $f\in L^2_\omega(\mathbb{R}^n) \cap L^p_\omega(\mathbb{R}^n)$.  First, $T_{\delta, R}f \in D(-\Delta_\omega)^{\alpha})$, so we have \begin{equation*}
        \| T_{\delta, R}f\|_{L^q_{{\omega}^{q/p}}(\mathbb{R}^n)} \leq C \| (-\Delta_\omega)^\alpha T_{\delta, R}f \|_{L^p_{\omega}(\mathbb{R}^n)}.
    \end{equation*}
   Now, $(-\Delta_\omega )^\alpha T_{\delta, R}f=\varphi_{\delta, R}(-\Delta_\omega )f$ for a holomorphic function $\varphi_{\delta, R}$ that is bounded on sectors $S_\mu$ uniformly in $\delta, R$. Thus, using Lemma \ref{lem:DR}, we obtain
    \begin{equation}\label{eq:checkRH}
        \left\| \int_{\delta}^{R}t^\alpha e^{t\Delta_\omega}f \, \frac{\mathrm{d}t}{t} \right\|_{L^q_{{\omega}^{q/p}}(\mathbb{R}^n)} \leq C \, \Gamma(\alpha)  \left\| f \right\|_{L^p_{\omega}(\mathbb{R}^n)}.
    \end{equation}
    Finally, to prove that $\omega \in RH_{\frac{q}{p}}(\mathbb{R}^n)$, we fix an arbitrary cube $Q$ and denote its side length by $\ell$. In the inequality \eqref{eq:checkRH}, we choose $\delta=\ell^2$, $R=2\ell^2$ and $f=\mathbb{1}_Q$. From the lower bound on the heat kernel in Lemma \ref{lem:GULB}, we  obtain as before
    $$  \int_{Q} \omega^{q/p}(x) \, \mathrm d x \leq \Tilde{C} |Q|^{1-\frac{q}{p}} \omega(Q)^{q/p}, $$
    where $\Tilde{C}<\infty$ is a positive constant depending only on $C$, $\alpha$, $p$, $q$, $n$, $D$ and the choice of $\mu$. Therefore, $\omega \in RH_{\frac{q}{p}}(\mathbb{R}^n)$.
\end{proof}

\section{Comparison with related results}

This brief section aims to convince the reader that our results are of a different nature to existing results in the literature. The ones imply the others partially and under restrictive conditions {and  the implications involve more complex tools.}

As we said before, the Sobolev embeddings for the unweighted Laplacian are equivalent to the classical Hardy-Littlewood-Sobolev inequalities. In \cite{muckenhoupt1974weighted}, Muckenhoupt and Wheeden extend this embeddings for the unweighted Laplacian to weighted spaces.

For all $\beta>0$, the Riesz potential $I_\beta$ is defined for suitable functions as 
\begin{equation*}
    I_\beta(f)(x) =\int_{\mathbb{R}^n} \frac{f(y)}{ \ \ |x-y|^{n-\beta }} \ \mathrm{d}y.
\end{equation*}
We have the classical equality $(-\Delta)^{-\alpha}f=c_{n,\alpha}I_{2\alpha} f$ with $c_{n,\alpha}$ is a positive constant depending on $n$ and $\alpha$, again for suitable $f$. In what follows, think of $\alpha=\beta/2$.
\begin{thm}[\cite{muckenhoupt1974weighted}]\label{thm:MW}
    Let $\Tilde{\omega}$ be a weight. Let $\beta \in (0,n)$, $p \in [1,\frac{n}{\beta})$ and $q:=\frac{np}{n-\beta p}$, that is, $\frac{1}{p}-\frac{1}{q}=\frac{\beta}{n}$. Then the following is true.
    \begin{enumerate}
        \item If $p>1$, then $I_\beta$ is bounded from $L^p_{\Tilde{w}^p}(\mathbb{R}^n)$ to $L^q_{\Tilde{w}^q}(\mathbb{R}^n)$ if and only if $\Tilde{\omega} \in A_{p,q}(\mathbb{R}^n)$.
        \item If $p=1$, then $I_\beta$ is bounded from $L^1_{\Tilde{w}}(\mathbb{R}^n)$ to $L^{q,\infty}_{\Tilde{w}^q}(\mathbb{R}^n)$ if and only if $\Tilde{\omega} \in A_{1,q}(\mathbb{R}^n)$.
    \end{enumerate}
\end{thm}

For $\alpha=1/2$, then $I_1$ can be used  to, roughly speaking, invert of the gradient and by construction $\| \nabla f \|_{L^2_\omega(\mathbb{R}^n)}= \| (-\Delta_\omega )^{1/2 } f \|_{L^2_\omega(\mathbb{R}^n)}$
on $H^1_\omega(\mathbb{R}^n)$. So, we may
 also make use of the following result concerning the boundedness of the Riesz transform, originally established by Cruz-Uribe, Martell, and Rios \cite{cruz2017kato}. 

\begin{thm}[\cite{cruz2017kato}]\label{thm:TRiesz}
 The Riesz transform $\nabla (-\Delta_\omega)^{-1/2}$ is of weak type $(1,1)$  and  is of strong type $(p,p)$  for all $1< p < q_+(-\Delta_\omega)$, where this last exponent is greater than 2.
\end{thm}

Using these results, we see that one can obtain similar inequalities, only for $0<\alpha\le 1/2$, in a smaller range of $p$ and assuming the weight in a smaller Muckenhoupt class and in dimension $n\ge 3$.

\begin{cor}\label{cor:MW}
     Let $p \in (1,2]$ and assume that $n \ge 3$ and $\omega \in A_p(\mathbb{R}^n) \cap RH_{\frac{n}{n-2}}(\mathbb{R}^n)$. Then there exists a constant $C = C([\omega]_{A_p}, {[\omega]_{RH_{\frac{n}{n-p}}}}, n) > 0$ such that for all $\alpha \in [0,1]$, setting $q := \frac{np}{n - p\alpha}$, we have
    \begin{equation}\label{eq:interpolationMW}
        \| f \|_{L^{q}_{{\omega}^{q/p}}(\mathbb{R}^n)} \leq C \| (-\Delta_\omega )^{\alpha/2 } f \|_{L^p_\omega(\mathbb{R}^n)}, \quad \text{for all} \ f \in D((-\Delta_\omega)^{\alpha/2 }).
    \end{equation}
    In particular, $(-\Delta_\omega)^{-\alpha/2}$ extends to a bounded operator from $L^p_\omega(\mathbb{R}^n)$ to $L^q_{\omega^{q/p}}(\mathbb{R}^n)$. { Furthermore, if $p\in [2,\min(n,q_+(-\Delta_\omega))$ then same result holds, provided that $\omega \in A_2(\mathbb{R}^n) \cap RH_{\frac{n}{n-p}}(\mathbb{R}^n)$ and the constant $C$ depends only on $n$, $[\omega]_{A_2}$, $[\omega]_{RH_{\frac{n}{n-p}}}$ and $n$.}
\end{cor}
\begin{proof}
    {We begin with the case $p\in (1,2]$}. We set $p^\star:=\frac{np}{n-p}$. By assumption, we have $\omega \in A_p(\mathbb{R}^n) \cap RH_{\frac{2^\star}{2}}(\mathbb{R}^n)$. In particular, we have $\omega \in A_p(\mathbb{R}^n) \cap RH_{\frac{p^\star}{p}}(\mathbb{R}^n)$ and $\omega \in A_2(\mathbb{R}^n) \cap RH_{\frac{2^\star}{2}}(\mathbb{R}^n)$. Using point (8) of Proposition \ref{prop:weights}, this is equivalent of saying that $\omega^{\frac{1}{p}}\in A_{p,p^\star}(\mathbb{R}^n)$ and  $\omega^{\frac{1}{2}}\in A_{2,2^\star}(\mathbb{R}^n)$. 
    
    {The second condition is used qualitatively as follows. We have the pointwise inequality $|f|\leq C(n) |I_1(\nabla f)|$  for all $f\in \mathcal{D}(\mathbb{R}^n)$. Using the density of $\mathcal{D}(\mathbb{R}^n)$ in $H^1_\omega(\mathbb{R}^n)$ together with Theorem \ref{thm:MW} with $\alpha=1$, {p=2} and $\Tilde{\omega}=\omega^{\frac{1}{2}}$, we deduce that  {for all $f\in H^1_\omega(\mathbb{R}^n)$}, $f$ and $I_1(\nabla f)$ are almost everywhere limits of sequences $f_k$ and $I_1(\nabla f_k)$ with $f_k\in \mathcal{D}(\mathbb{R}^n)$. Hence  $|f|\leq C(n) |I_1(\nabla f)|$ almost everywhere on $\mathbb{R}^n$, for all $f\in D((-\Delta_\omega)^{1/2})=H^1_\omega(\mathbb{R}^n)$. 
   
   Next, this inequality and Theorem \ref{thm:MW} with $\alpha=1$ and $\Tilde{\omega}=\omega^{\frac{1}{p}}$  imply that
    \begin{equation*}
        \left\| f\right\|_{L^{p^\star}_{\omega^{p^\star/p}}(\mathbb{R}^n)} \leq C(n) \left\| I_1(\nabla f)\right\|_{L^{p^\star}_{\omega^{p^\star/p}}(\mathbb{R}^n)}\leq {C( n,p,[\omega]_{A_p},[\omega]_{RH_{\frac{n}{n-p}}} )}  \|\nabla f \|_{L^p_\omega(\mathbb{R}^n)},
    \end{equation*}
    for all $f \in D((-\Delta_\omega)^{1/2})$. Using Theorem \ref{thm:TRiesz}, we deduce that for all $f \in D((-\Delta_\omega)^{1/2})$,
    \begin{equation}\label{system:interpolation}
        \| f \|_{L^{p^\star}_{\omega^{p^\star/p}}(\mathbb{R}^n)} \leq C  \|(-\Delta_\omega)^{1/2} f \|_{L^p_\omega(\mathbb{R}^n)}.
    \end{equation} 
    Now, we rewrite \eqref{system:interpolation}, for all $f \in D((-\Delta_\omega)^{-1/2}) \cap L^p_\omega(\mathbb{R}^n)$, as
    \begin{align}\label{eq:Voigt}
       \| (-\Delta_\omega)^{-1/2} f \|_{L^{p^\star}_{\omega^{p^\star/p}}(\mathbb{R}^n)} \leq C  \| f \|_{L^p_\omega(\mathbb{R}^n)}.
     \end{align} 
   At this point,  we apply an abstract version of Stein's interpolation theorem due to Voigt \cite{Voigt1992} to the analytic family $\Psi(z) := (-\Delta_\omega)^{-\frac{z}{2}}$ defined for $0 \le \mathrm{Re}(z) \le 1$, proceeding as follows.
   By the boundedness of the $H^\infty$-calculus on $L^p_\omega(\mathbb{R}^n)$ in Lemma \ref{lem:DR}, $-\Delta_\omega$ has bounded imaginary powers. {This, together with inequality \eqref{eq:Voigt}}
   give us that for all $f\in D((-\Delta_\omega)^{-1/2}) \cap L^p_\omega(\mathbb{R}^n)$, 
    \begin{align*}  
        \sup_{t\in \mathbb{R}}\|  \Psi(it)f \|_{L^p_\omega(\mathbb{R}^n)} &
        \le c_0 \| f \|_{L^p_\omega(\mathbb{R}^n)},
   \\
   \sup_{t\in \mathbb{R}} \|  \Psi(1+it)f \|_{L^{p^\star}_{\omega^{p^\star/p}}(\mathbb{R}^n)} &
   \leq c_1 \| f \|_{L^p_\omega(\mathbb{R}^n)},
    \end{align*}
    where $c_0<\infty$ and $c_1<\infty$ are finite constants that are independent of $f$ and  that depend only on $n$, $p$, $[\omega]_{A_p}$ and $[\omega]_{RH_{\frac{n}{n-p}}}$. Using this abstract version of Stein's interpolation theorem \cite{Voigt1992}, for all $\theta \in (0,1)$ and $f\in D((-\Delta_\omega)^{-1/2}) \cap L^p_\omega(\mathbb{R}^n)$, we have 
    \begin{equation*}
        \| \Psi(\theta)f \|_{[ L^p_\omega(\mathbb{R}^n),L^{p^\star}_{\omega^{p^\star/p}}(\mathbb{R}^n) ]_\theta } \le c_0^{1-\theta} c_1^\theta \, \| f \|_{L^p_\omega(\mathbb{R}^n)}.
    \end{equation*}
    Choosing $\theta \in (0,1)$ such that $\frac{1}{q}=\frac{1-\theta}{p}+\frac{\theta}{p^\star}$ and using the Stein-Weiss complex interpolation theorem with a change of measure (see \cite[Theorem 5.5.3]{bergh2012interpolation}), we have 
     $[ L^p_\omega(\mathbb{R}^n),L^{p^\star}_{\omega^{p^\star/p}}(\mathbb{R}^n) ]_\theta = L^{q}_{{\omega}^{q/p}}(\mathbb{R}^n),$
    and we obtain
     \begin{equation}\label{eq:MW}
        \| (-\Delta_\omega )^{-\alpha/2 } f \|_{L^{q}_{{\omega}^{q/p}}(\mathbb{R}^n)} \leq C \|  f \|_{L^p_\omega(\mathbb{R}^n)}, \quad \text{for all} \ f \in D((-\Delta_\omega)^{-1/2 })\cap L^p_\omega(\mathbb{R}^n).
    \end{equation}
    By density, as in step 4 of the proof of Theorem \ref{thm:principal} this shows that $(-\Delta_\omega)^{-\alpha/2}$ defines a bounded operator from $L^p_\omega(\mathbb{R}^n)$ to $L^q_{\omega^{q/p}}(\mathbb{R}^n)$. Finally, the above inequality also holds for $f \in D((-\Delta_\omega)^{-\alpha/2 })\cap L^p_\omega(\mathbb{R}^n)$ and we use that $D((-\Delta_\omega)^{-\alpha/2 })$ is the range of $(-\Delta_\omega)^{\alpha/2 }$ to obtain \eqref{eq:interpolationMW}. 
    
    {The case $p\in [2,\min(n,q_+(-\Delta_\omega))$ is treated similarly and we skip details.}}
\end{proof}

\begin{rem} 
  We could have used \cite[Theorem 2.7]{lacey2010sharp} instead of Theorem~\ref{thm:MW}. The former provides a weighted Sobolev-Gagliardo-Nirenberg  inequality involving the gradient, and it also includes the endpoint case $p = 1$. This inequality is also sharp in the control of the bound with respect to the weight. More precisely, \cite[Theorem 2.7]{lacey2010sharp} establishes that for any $p \geq 1$ and any weight $\tilde{\omega} \in A_{p,q}(\mathbb{R}^n)$, where $q<\infty$ is such that $\frac{1}{p} - \frac{1}{q} = \frac{1}{n}$, the following estimate holds for every Lipschitz function $f$ with compact support:
  \begin{equation*}
  \| f \|_{L^q_{{\tilde\omega}^{q}}(\mathbb{R}^n)} \leq {c(n,p)} [\Tilde{\omega}]_{A_{p,q}}^{1/n'} \| \nabla f \|_{L^p_{\tilde\omega^p}(\mathbb{R}^n)}.
  \end{equation*}
  Remark that one must assume $n \geq 3$ to use it when $p = 2$. In another direction, one can use \cite[Proposition~6.1]{cruz2017kato} which provides a bound $\| (-\Delta_\omega )^{1/2 } f \|_{L^p_{\omega}(\mathbb{R}^n)} \lesssim \| \nabla f \|_{L^p_{\omega}(\mathbb{R}^n)}$ when $p\in (r_\omega, \infty)$ where $r_\omega<2$ is the infimum of exponents {$r\ge 1$ for which $w\in A_r(\mathbb{R}^n)$}. Our result, together with this inequality, implies the above inequality in  this range of $p$ and no information on the control of the constant. \end{rem}

\begin{rem} We have not tried to write down the kernels of $(-\Delta_\omega)^{-\alpha}$ as we do not know how to exploit them in general. In some cases, they can be estimated in order to use other results (still necessarily with non optimal conditions). Here is one instance.
    Define the reverse doubling  exponent of $\omega$, denoted by $\mathrm{RD}$, as the largest positive constant such that
    \begin{equation*}
    \omega(B(x,r)) \leq 2^{-\mathrm{RD}} \omega(B(x,2r)), \quad \text{for all } x \in \mathbb{R}^n \text{ and } r > 0.
    \end{equation*}
    The existence of such a constant is ensured by \eqref{MuckProportion}. Using the Gaussian upper and lower bounds for the heat kernel given in Lemma \ref{lem:GULB}, together with the representation formula
    \begin{equation*}
    (-\Delta_\omega)^{-\alpha}f = \frac{1}{\Gamma(\alpha)} \int_{0}^{\infty} t^{\alpha} e^{t\Delta_\omega}f \, \frac{\mathrm{d}t}{t},
    \end{equation*}
    we can show that, whenever $\alpha < \frac{\mathrm{RD}}{2}$, the following pointwise estimate for nonnegative functions holds
    \begin{equation*}
    (-\Delta_\omega)^{-\alpha}f(x) \approx \int_{\mathbb{R}^n} \frac{|x - y|^{2\alpha}}{\omega(B(x, |x - y|))} f(y) \, \mathrm{d}\omega(y).
    \end{equation*}
    These weighted Riesz potentials are discussed in \cite[section 5]{auscher2008weighted}. However, further restrictive conditions on $\omega$ must be assumed in order to apply the results obtained therein.
\end{rem}

\section{Case of degenerate elliptic operators with coefficients}\label{section: coeff}
In this section, we fix a matrix-valued function $A:  \mathbb{R}^n \rightarrow M_n(\mathbb{C})$ with complex measurable coefficients and such that 
\begin{equation*}
\left | A(x)\xi \cdot \zeta  \right |\leq M \omega(x)\left | \xi \right |\left | \zeta  \right |, \ \ \ \  
\nu \left | \xi \right |^2\omega(x)\leq \mathrm{Re}( A(x)\xi\cdot \overline{\xi})
\end{equation*}
for some $\nu>0, M<\infty$ and for all $\zeta, \xi \in \mathbb{C}^n$ and all $x \in \mathbb{R}^n$.

We define $L_\omega$ as the unbounded operator on $L^2_\omega(\mathbb{R}^n)$ associated with the following bounded, elliptic, and sectorial sesquilinear form on $H^1_\omega(\mathbb{R}^n) \times H^1_\omega(\mathbb{R}^n)$ defined by
\begin{equation*}
    (u,v) \mapsto \int_{\mathbb{R}^n} A(x)\nabla_x u \cdot \overline{\nabla_x v}  \ \mathrm d \omega.
\end{equation*}
The operator $L_\omega$ is sectorial with sectoriality angle $\varphi \in [ \arctan\left(\frac{M}{\nu}\right), \frac \pi 2)$, and it is injective since the measure $\mathrm{d}\omega$ has infinite mass by \eqref{MuckProportion}. For all $\beta \in \mathbb{R}$, we define $L_\omega^\beta$ as the closed operator $\mathbf{z^\beta}(L_\omega)$, constructed via the functional calculus for sectorial operators (see \cite{haase2006functional}). We denote the domain of $L_\omega^\beta$ by $D(L_\omega^\beta)$.

When the coefficients of $A$ are real-valued, all the previously stated results concerning the degenerate Laplacian $-\Delta_\omega$, such as the Gaussian upper (and lower) bounds for the heat kernel and the boundedness of the $H^\infty$-calculus on $L^p_\omega(\mathbb{R}^n)$ for every $p \in (1, \infty)$ with the same references {(see \cite{ataei2024fundamental, baadi2025degenerate} and \cite[Theorem 3.1]{duong1996semigroup}, respectively)} remain valid for the degenerate second-order operator $L_\omega$, with the only modification being the sectoriality angle, which changes from $0$ to $\varphi$. We may therefore adapt results that just need the Gaussian upper bounds to hold  by replacing $-\Delta_\omega$ with $L_\omega$. As an illustration, the analog of  Theorem \ref{thm:principal*} and Theorem \ref{cor:principal**} is as follows.

\begin{thm}\label{thm:principal avec coeff}
Let $1<p<q<\infty$. Assume that $\omega \in RH_{\frac{q}{p}}(\mathbb{R}^n)$ and set $\alpha=\frac{n}{2}(\frac{1}{p}-\frac{1}{q})$. Then, there exists a constant $C = C([ \omega  ]_{A_2},[ \omega  ]_{RH_{\frac{q}{p}}},n,p,q)<\infty$ such that
    \begin{equation*}
     \|w^{1/p} f \|_{L^{q,p}_{\mathrm{d}x}(\mathbb{R}^n)} +   \| f \|_{L^{q,p}_{\omega}(\omega^{\frac{1}{p}-\frac{1}{q}}, \mathbb{R}^n)} \leq C \| L_\omega^\alpha f \|_{L^p_{\omega}(\mathbb{R}^n)}, \quad \forall f \in D(L_\omega^\alpha).
    \end{equation*}
    In particular, $L_\omega^{-\alpha}$ defines a bounded operator from $L^p_\omega(\mathbb{R}^n)$ to $L^{q,p}_{\omega}(\omega^{\frac{1}{p}-\frac{1}{q}}, \mathbb{R}^n)$ and $w^{1/p}L_\omega^{-\alpha}$ is bounded from $L^p_\omega(\mathbb{R}^n)$ to $L^{q,p}_{\mathrm{d}x}(\mathbb{R}^n)$. {Moreover, when $p=1$ and $\omega\in RH_{{q}}(\mathbb{R}^n)$, then
    \begin{equation*}
        \| \omega  f \|_{L^{q,\infty}_{\mathrm{d}x}(\mathbb{R}^n)} \leq C \| L_\omega^{\alpha}  f \|_{L^1_{\omega}(\mathbb{R}^n)}, \quad \forall f \in D(L_\omega^\alpha).
    \end{equation*}
    In particular, $\omega \, L_\omega^{-\alpha}$ defines a bounded operator from $L^1_\omega(\mathbb{R}^n)$ to $L^{q,\infty}_{\mathrm{d}x}(\mathbb{R}^n)$.}
\end{thm}

When the coefficients are complex, the reference \cite{cruz2017kato} provides us with a thorough description of the properties of $L_\omega$. In particular, it shows the existence of an interval $(p_-(L_\omega), p_+(L_\omega)) \subset (1,\infty)$, which is the largest open interval containing $2$ of exponents $p$ for which the semigroup $e^{-tL_\omega}$ is uniformly bounded on $L^p_\omega(\mathbb{R}^n)$. One has $p_-(-\Delta_\omega)=1$ and $p_+(-\Delta_\omega)=\infty$. The Gaussian upper bounds are replaced by a notion called $L^p_\omega-L^q_\omega$ off-diagonal estimates on balls (which can equivalently be replaced by cubes), denoted by $\mathcal{O}(L^p_\omega - L^q_\omega)$. See Definition~2.33 and Proposition~3.1 in that reference. It is reasonable to think that there is a result corresponding to the first part of Theorem \ref{thm:principal avec coeff} in the range $p_-(L_\omega)<p<q< p_+(L_\omega)$. However, it is not clear how to adapt Lemma \ref{lem:GULB} using this notion, which comes with a different point of view. We interpret weighted spaces via multiplication by the weight, which allows us to extend techniques from the unweighted theory. In particular, the target space is endowed with the measure $\omega^{q/p}\mathrm{d}x$. In contrast, in \cite{cruz2017kato}, the results are formulated in the space of homogeneous type  $(\mathbb{R}^n,\mathrm{d}w,|\cdot|)$.

We mention that, even in the absence of Gaussian upper bounds, one may still obtain some results from \cite{muckenhoupt1974weighted,lacey2010sharp}, at the expense of strong hypotheses. More precisely, one can reproduce the argument in Corollary~\ref{cor:MW} to obtain \eqref{system:interpolation} with $-\Delta_\omega$ replaced by $L_\omega$, under the assumptions that $n \ge 3$, $p_-(L_\omega)<p<q_+(L_\omega)$ and $\omega \in A_p(\mathbb{R}^n) \cap RH_{\frac{n}{n-2}}(\mathbb{R}^n)$. In addition,  one needs the solution to the Kato problem for $L_\omega$ in \cite{cruz2015kato} and its $L^p_\omega$ generalization in \cite{cruz2017kato} before applying the subsequent interpolation argument.

Another interesting remark is that when $p=2$ and $\alpha   \in (0,1/2)$, one can make use of the results just established for $(-\Delta_\omega)^\alpha$. Indeed, Kato's theorem \cite{Kato61} yields $D(L_\omega^\alpha) = D((-\Delta_\omega)^\alpha)$, and the norms $\| L_\omega^\alpha \cdot \|_{2,\omega}$ and $\| (-\Delta_\omega)^\alpha \cdot \|_{2,\omega}$ are equivalent, independently of the coefficients.

\subsubsection*{\textbf{Copyright}}
A CC-BY 4.0 \url{https://creativecommons.org/licenses/by/4.0/} public copyright license has been applied by the authors to the present document and will be applied to all subsequent versions up to the Author Accepted Manuscript arising from this submission.

\bibliographystyle{alpha}
\bibliography{references.bib}

\end{document}